\documentclass[reqno,a4paper,11pt]{amsart}

\usepackage[pdfpagelabels]{hyperref}
\usepackage{amssymb}
\usepackage{enumitem}
\usepackage{url}

\DeclareFontFamily{U}{mathx}{\hyphenchar\font45}
\DeclareFontShape{U}{mathx}{m}{n}{
      <5> <6> <7> <8> <9> <10>
      <10.95> <12> <14.4> <17.28> <20.74> <24.88>
      mathx10
      }{}
\DeclareSymbolFont{mathx}{U}{mathx}{m}{n}
\DeclareFontSubstitution{U}{mathx}{m}{n}
\DeclareMathAccent{\widecheck}      {0}{mathx}{"71}

\newtheorem{theorem}{Theorem}
\newtheorem{lemma}[theorem]{Lemma}
\newtheorem{corollary}[theorem]{Corollary}

\theoremstyle{definition}

\theoremstyle{remark}

\newtheorem{example}{Example}

\newcommand{\D}{\mathbb{D}}

\newcommand{\N}{\mathbb{N}}

\newcommand{\R}{\mathbb{R}}

\newcommand{\C}{\mathbb{C}}

\renewcommand{\phi}{\varphi}
\DeclareMathOperator{\Real}{Re}

\DeclareMathOperator{\mes}{mes}

\begin{document}

\title[Description of growth and oscillation of solutions]{Description of growth and oscillation of solutions of complex LDE's}
\thanks{The second author was supported in part by the Academy of Finland project \#286877.
The fourth author is supported in part by Ministerio de Econom\'ia y Competitivivad, Spain, 
projects MTM2014-52865-P and MTM2015-69323-REDT; and La Junta de Andaluc\'ia, project FQM210.}

\author{I.~Chyzhykov}
\address{Faculty of Mathematics and Computer Science\newline
\indent Warmia and Mazury University of Olsztyn\newline
\indent S\l{}oneczna 54, Olsztyn, 10710, Poland}               
\email{chyzhykov@matman.uwm.edu.pl}

\author{J.~Gr\"ohn}
\address{Department of Physics and Mathematics\newline
\indent University of Eastern Finland\newline 
\indent P.O. Box 111, FI-80101 Joensuu, Finland}
\email{\vspace*{-0.25cm}janne.grohn@uef.fi}

\author{J.~Heittokangas}
\email{\vspace*{-0.25cm} janne.heittokangas@uef.fi}

\author{J.~R\"atty\"a}
\email{jouni.rattya@uef.fi}

\date{\today}




\begin{abstract}
It is known that,
equally well in the unit disc as in the whole complex plane, the growth of the analytic 
coefficients $A_0,\dotsc,A_{k-2}$ of
\begin{equation*} 
f^{(k)} + A_{k-2} f^{(k-2)} + \dotsb + A_1 f'+ A_0 f = 0, \quad k\geq 2,
\end{equation*}
determines, under certain growth restrictions, not only the growth but also the oscillation of its non-trivial solutions,
and vice versa.
A~uniform treatment of this principle is given in the 
disc $D(0,R)$, $0<R\leq \infty$,
by using several measures for growth that are more flexible than those in the existing literature, 
and therefore permit more detailed analysis. In particular, results obtained
are not restricted to
cases where solutions are of finite (iterated) order of growth in the classical sense. The new findings are 
based on an~accurate integrated estimate for logarithmic derivatives of meromorphic functions,
which preserves generality in terms of three free parameters.
\end{abstract}

\maketitle


\section{Introduction}
It is a~well-known fact that the growth of
analytic coefficients $A_0, \dotsc, A_{k-1}$ of the differential equation
\begin{equation} \label{de:preli}
  f^{(k)} + A_{k-1} f^{(k-1)} + \dotsb + A_1 f'+ A_0 f = 0, \quad k\geq 2,
\end{equation} 
restricts the growth of solutions of \eqref{de:preli}, and vice versa.
Here we assume analyticity in the disc $D(0,R)$, where $0<R\leq \infty$. We write $\D = D(0,1)$ and $\C = D(0,\infty)$ for short.
In the case $A_{k-1} \equiv 0$ the oscillation of non-trivial solutions of \eqref{de:preli}
provides a~third property, which is known to be equivalent to the other two in certain cases \cite{HR:2011}, \cite{PR:2014}.
Recall also that there exists a~standard transformation which yields $A_{k-1}\equiv 0$
and leaves the zeros of solutions invariant; see \cite{HR:2011} and \cite[p.~74]{L:1993}.

In the present paper we content ourselves to
the case $A_{k-1}\equiv 0$.
Our intention is to elaborate on new circumstances
in which the growth of the Nevanlinna functions $T(r,f)$ and $N(r,1/f)$ of any non-trivial solution $f$ of \eqref{de:preli} and
the growth of the quantity
\begin{equation} \label{eq:qua}
  \max_{j=0,\dotsc, k-2} \,\int_{D(0,r)} |A_j(z)|^{\frac{1}{k-j}} \, dm(z)
\end{equation}
are interchangeable in an~appropriate sense. 
By the growth estimates for solutions of linear differential equations \cite{HKR:2004},
we deduce the asymptotic inequalities
\begin{equation*}
  N(r,1/f) \lesssim 1+T(r,f) \lesssim 1 + \sum_{j=0}^{k-2} \int_{D(0,r)} |A_j(z)|^{\frac{1}{k-j}} \, dm(z),
\end{equation*}
where the comparison constants depend on the initial values of $f$.
Therefore the problem at hand reduces to showing that, if $N(r,1/f)$ of any non-trivial solution~$f$ 
of \eqref{de:preli} has a~certain growth rate, then the quantity in \eqref{eq:qua} has the same or similar growth rate.
An outline of the proof is as follows.
The growth of Nevanlinna characteristics of quotients of linearly independent solutions
can be controlled by the second main theorem of Nevanlinna and the assumption
on zeros of solutions. The~classical representation theorem \cite{K:1969}
provides us means to express coefficients in terms of quotients of linearly independent solutions.
Since this representation entails logarithmic derivatives of meromorphic functions, 
this argument boils down to establishing accurate integrated logarithmic derivative estimates involving several free parameters.

One of the benefits of our approach
on differential equations is the freedom provided by various growth indicators.
This allows us to treat a~large scale of growth categories by uniform generic statements. In particular, results obtained
are not restricted to
cases where solutions are of finite (iterated) order of growth in the classical sense.
The other advantage is the fact that both cases of the whole complex plane and the finite disc can be covered simultaneously.

Logarithmic derivatives of meromorphic functions are considered from a~new perspective which 
preserves generality in terms of three free parameters.
Indeed, assuming that $f$ is meromorphic in a~domain containing the closure $\overline{D(0,R)}$, we estimate
area integrals of generalized logarithmic derivatives of the type
\begin{equation*}
  \int_{r'<|z|<r}  \bigg| \frac{f^{(k)}(z)}{f^{(j)}(z)}  \bigg|^{\frac{1}{k-j}} \, dm(z),
\end{equation*}
where $r'<r<R$ are free, and no exceptional set occurs. Such estimates are of course also of independent interest. Our findings
are accurate, as demonstrated by concrete examples, and
improve results in the existing literature.

The remainder of this paper is organized as follows. The results on differential equations 
and on logarithmic derivatives are discussed in Sections~\ref{sec:results_DE} and~\ref{sec:logdva}, 
respectively. Results on logarithmic derivatives  are proved in
Sections~\ref{sec:logdva_proof} and \ref{sec:nn}, 
while the proofs of the results on differential equations are presented in Sections~\ref{sec:jh}--\ref{sec:mem}.


\section{Results on differential equations} \label{sec:results_DE}

Let $0<R\leq \infty$ and $\omega\in L^1(0,R)$.
The extension defined by 
$\omega(z) = \omega(|z|)$ for all $z\in D(0,R)$ is called a~radial weight on $D(0,R)$.
For such an $\omega$, write
$\widehat{\omega}(z) = \int_{|z|}^R \omega(s)\, ds$ for $z\in D(0,R)$. We assume throughout 
the paper that $\widehat{\omega}$ is strictly positive on $[0,R)$, for otherwise $\omega(r)=0$ 
for almost all $r$ close to~$R$, and that case is not interesting in our setting.

Our first result characterizes differential equations 
\begin{equation} \label{ldek_gen}
  f^{(k)} + A_{k-2} f^{(k-2)} + \dotsb + A_1 f'+ A_0 f = 0, \quad k\geq 2,
\end{equation} 
whose solutions belong to a Bergman-Nevanlinna type space \cite{N:1925}, \cite{PR:2014}. 
The novelty of this result 
does not only stem from the general growth indicator induced by the auxiliary functions $\Psi,\omega,s$
but also lies in the fact that it includes the cases of the finite disc and the whole complex plane in a~single result.


\begin{theorem} \label{thm:psi}
Let $\Psi:[0,\infty) \to [0,\infty)$ be a~non-decreasing function which satisfies
$\Psi(x^2) \lesssim \Psi(x)$ for all $0\leq x< \infty$, and $\Psi(\log x) = o(\Psi(x))$ as $x\to\infty$. For fixed $0<R\leq \infty$,
let $s:[0,R) \to [0,R)$ be an increasing function such
that $s(r)\in (r,R)$ for all $0\leq r <R$, let $\omega$ be a~radial weight such that 
$\widehat{\omega}(r) \lesssim  \widehat{\omega}(s(r))$ for all $0\leq r <R$, and assume
\begin{equation} \label{eq:intass}
  \int_0^R \Psi\bigg(  s(r) \log \frac{e \, s(r)}{s(r)-r}  \bigg) \, \omega(r) \, dr<\infty.
\end{equation}
If the coefficients $A_0,\dotsc,A_{k-2}$ are analytic in $D(0,R)$, then the following conditions are equivalent:
\begin{enumerate}[leftmargin=0.75cm]
\item[\rm (i)]
  $\displaystyle\int_0^R \Psi\bigg( \int_{D(0,r)} |A_j(z)|^{\frac{1}{k-j}} \, dm(z) \bigg) \, \omega(r)\, dr < \infty$
  for all $j=0,\dotsc,k-2$;

\item[\rm (ii)]
  $\displaystyle\int_0^R \Psi\big( T(r,f) \big) \, \omega(r) \, dr < \infty$ for all solutions $f$ of \eqref{ldek_gen};

\item[\rm (iii)]
  $\displaystyle\int_0^R \Psi\big( N(r,1/f) \big) \, \omega(r) \, dr < \infty$ for all non-trivial solutions $f$ of \eqref{ldek_gen}.
\end{enumerate}
\end{theorem}

Note the following observations regarding Theorem~\ref{thm:psi}:

(a) The analogues of (i) and (ii) are equivalent 
also for the differential equation~\eqref{de:preli}. See \cite{CS:2016} for another general scale to measure the growth 
in the case of the complex plane. 

(b) The result is relevant only when $\Psi$ is unbounded.

(c) The classical choices for $s$
in the cases of $D(0,R)$ and $\C$ are $s(r)=(r+R)/2$ and $s(r)=2r$, respectively. 
While the function~$s$ is absent in the assertions (i)--(iii),
its effect is implicit through the dependence in the hypothesis on $s$, $\Psi$ and~$\omega$. In terms of
applications, the auxiliary function $s$ provides significant freedom to possible choices of $\Psi$ and $\omega$.

(d) The condition $\Psi(x^2) \lesssim \Psi(x)$ requires slow growth and local smoothness. For example, it is satisfied by any positive
power of any (iterated) logarithm. To see that restrictions on the growth alone do not
imply this condition, let $g$ be any non-decreasing unbounded function. Choose
a~sequence $\{x_j\}_{j=1}^{\infty}$ such that $g(x_j) \geq 2^{2^j}$ and $x_{j+1} \geq x_j^2$,
and define $h$ such that $h(x)=2^{2^j}$ for $x_j\leq x < x_{j+1}$. Then $g$ dominates $h$, while $h(x_n)/h(\sqrt{x_n})=2^{2^{n-1}}\to \infty$ as $n\to\infty$.

(e) The assumption $\Psi(\log x) = o(\Psi(x))$, as $x\to\infty$, is trivial for typical choices of $\Psi$
such as $\Psi(x)=\log^+ x$.
However, the condition
is not satisfied by all continuous, increasing and unbounded functions $\Psi$.
A~counterexample is given by $\Psi(x)=\log_n(x)+(n-1)(e-1)$,
$e_n(1)\leq x\leq e_{n+1}(1)$, for which $\Psi(\log x) \sim \Psi(x)$ as $x\to\infty$.
Here
$\log_n$ and $e_n$ stand for iterative logarithms and exponentials, respectively.

(f) For a~fixed $s$, the requirement $\widehat{\omega}(r) \lesssim  \widehat{\omega}(s(r))$
not only controls the rate at which $\widehat\omega$ decays to zero but also demands certain local smoothness.
The situation is in some sense similar to that of $\Psi$.

(g) Theorem~\ref{thm:psi} is relevant only when some solution $f$ of \eqref{ldek_gen} satisfies
\begin{equation*}
\limsup_{r\to R^-} \, \frac{T(r,f)}{s(r)\log((e s(r))/(s(r)-r))} = \infty,
\end{equation*}
but its applicability is not
restricted to any pregiven growth scale. Indeed,
if~$f$ is an~arbitrary
entire function, then we find a~sufficiently smooth and fast growing increasing function~$\varphi$
such that its growth exceeds that of $T(r,f)$ and its inverse~$\varphi^{-1}=\Psi$ satisfies 
$\Psi(x^2) \lesssim \Psi(x)$.
Further, if $s(r)=2r$ and $\omega(r)=(1+r)^{-3}$, then all requirements on $\Psi$, $\omega$ and 
$s$ are fulfilled, and
\begin{equation*}
  \int_0^\infty \Psi\big( T(r,f) \big) \, \omega(r) \, dr
  \leq \int_0^\infty \Psi(\phi(r))\, \omega(r) \, dr 
  = \int_0^\infty r\, \omega(r) \, dr< \infty.
\end{equation*}
The case of the finite disc is similar. This shows, in particular, that Theorem~\ref{thm:psi} is not restricted to functions of finite 
iterated order in the classical sense.

\medskip

Observations similar to (a)--(g) apply for forthcoming results also.

Arguments in the proof of Theorem~\ref{thm:psi} also apply in the case where growth
indicators given in terms of integrals are replaced with ones stated in terms of limit superiors.


\begin{theorem} \label{thm:psi2}
Let $\Psi:[0,\infty) \to [0,\infty)$ be a~non-decreasing function which satisfies
$\Psi(x^2) \lesssim \Psi(x)$ for all $0\leq x< \infty$, and $\Psi(\log x) = o(\Psi(x))$ as $x\to\infty$. For fixed $0<R\leq \infty$,
let $s:[0,R) \to [0,R)$ be an increasing function such
that $s(r)\in (r,R)$ for all $0\leq r <R$, let $\omega$ be a~radial weight such that 
$\widehat{\omega}(r) \lesssim  \widehat{\omega}(s(r))$ for all $0\leq r <R$, and assume
\begin{equation*} 
   \limsup_{r\to R^-} \, \Psi\bigg(  
   s(r) \log \frac{e \, s(r)}{s(r)-r} 
   \bigg) \, \widehat\omega(r)<\infty.
\end{equation*}
If the coefficients $A_0,\dotsc,A_{k-2}$ are analytic in $D(0,R)$, then the following conditions are equivalent:
\begin{enumerate}[leftmargin=0.75cm]
\item[\rm (i)]
  $\displaystyle \limsup_{r\to R^-} \,  \Psi\bigg( \int_{D(0,r)} |A_j(z)|^{\frac{1}{k-j}} \, dm(z) \bigg) \, \widehat\omega(r)\, < \infty$
  for all $j=0,\dotsc,k-2$;

\item[\rm (ii)]
  $\displaystyle \limsup_{r\to R^-} \,  \Psi\big( T(r,f) \big) \, \widehat\omega(r) < \infty$ for all solutions $f$ of \eqref{ldek_gen};

\item[\rm (iii)]
  $\displaystyle \limsup_{r\to R^-} \, \Psi\big( N(r,1/f) \big) \, \widehat\omega(r)  < \infty$ for all non-trivial solutions $f$ of \eqref{ldek_gen}.
\end{enumerate}
\end{theorem}

Proofs of Theorem~\ref{thm:psi} and \ref{thm:psi2} are similar and the latter is omitted.
The small-oh version of Theorem~\ref{thm:psi2} is also valid in the sense
that the finiteness of limit superiors can be replaced by the requirement that they are zero (all five of them).

Let $\widehat{\mathcal{D}}$ be the class of radial weights for which there exists a~constant
$C=C(\omega)\geq 1$ such that
$\widehat{\omega}(r) \leq C  \, \widehat{\omega}((1+r)/2)$ for all $0\leq r <1$.
Moreover, let $\widecheck{\mathcal{D}}$ be the class of radial weights for which there exist constants
$K=K(\omega)\geq 1$ and $L=L(\omega)\geq 1$ such that
$\widehat{\omega}(r) \geq L  \, \widehat{\omega}(1-(1-r)/K)$ for all $0\leq r <1$.
We write $\mathcal{D} = \widehat{\mathcal{D}} \cap \widecheck{\mathcal{D}}$ for brevity.
For a~radial weight $\omega$, define
\begin{equation*}
  \omega^\star(z) = \int_{|z|}^1 \omega(s) \, \log\frac{s}{|z|} \, s\, ds, \quad z\in\D\setminus\{0\}.
\end{equation*}

We proceed to consider an~improvement of the main result in \cite[Chapter~7]{PR:2014}, which
concerns \eqref{ldek_gen} in the unit disc. The following result
is a~far reaching generalization of \cite[Theorem~7.9]{PR:2014} 
requiring much less regularity on the weight~$\omega$.


\begin{theorem} \label{thm:ident}
Let $\omega\in \mathcal{D}$. If the coefficients $A_0,\dotsc,A_{k-2}$ are analytic in $\D$, 
then the following conditions are equivalent:

\begin{enumerate}[leftmargin=0.75cm]
\item[\rm (i)] $\displaystyle\int_{\D} |A_j(z)|^{\frac{1}{k-j}} \, \widehat{\omega}(z) \, dm(z) < \infty$ for all $j=0,\dotsc,k-2$;

\item[\rm (ii)] $\displaystyle{\int_0^1} T(r,f) \, \omega(r)\, dr < \infty$
for all solutions of \eqref{ldek_gen};

\item[\rm (iii)] $\displaystyle{\int_0^1} N(r,1/f) \, \omega(r)\, dr < \infty$
for all non-trivial solutions of \eqref{ldek_gen};

\item[\rm (iv)] zero sequences $\{z_k\}$ of non-trivial solutions of \eqref{ldek_gen} satisfy
$\displaystyle\sum_k \omega^\star(z_k)<\infty$.
\end{enumerate}
\end{theorem}

In Theorem~\ref{thm:ident} we may assume that possible value $z_k=0$ is removed from the zero-sequence.
Note that this result is not a~consequence of Theorem~\ref{thm:psi}, and vice versa. 
Roughly speaking Theorem~\ref{thm:ident} corresponds to the case $\Psi(x)=x$, which is excluded in 
Theorem~\ref{thm:psi}. 
Also Theorem~\ref{thm:psi} extends to cases which cannot be reached by \cite[Theorem~7.9]{PR:2014}.
We refer to the discussion in the end of \cite[Chapter~7]{PR:2014} for more details.

The counterpart of Theorem~\ref{thm:ident} for the complex plane 
is the case with polynomial coefficients, which is known by the existing literature~\cite{HR:2011}.
This is also the reason why Theorem~\ref{thm:ident} is restricted to~$\D$.

Our final result on differential equations 
is a~normed analogue of Theorem~\ref{thm:psi2}, and therefore its proof requires more detailed analysis.
It is based on another limsup-order, which is defined and discussed next.
Let $\Psi: [0,\infty)\to\R^+$ and $\varphi: (0,R) \to \R^+$ be continuous, increasing and unbounded functions,
where $0<R\leq\infty$. We define the $(\Psi,\varphi)$-order of a non-decreasing function $\psi: (0,R) \to \R^+$~by
\begin{equation*}
  \rho_{\Psi,\varphi}(\psi)=\limsup_{r\to R^-}\frac{\Psi(\log^+\psi(r))}{\log\varphi(r)}.
\end{equation*}
This generalizes the $\varphi$-order introduced in \cite{CHR:2009}.
If $f$ is meromorphic in $D(0,R)$, then the $(\Psi,\varphi)$-order of $f$ is defined
as $\rho_{\Psi,\varphi}(f)=\rho_{\Psi,\varphi}(T(r,f))$. If $a\in\widehat{\C}$, then
the $(\Psi,\varphi)$-exponent of convergence of the $a$-points of $f$ is defined as
$\lambda_{\Psi,\varphi}(a,f)=\rho_{\Psi,\varphi}(N(r,a,f))$.
These two concepts regarding $f$ reduce to the classical cases in the plane if $\Psi$ and $\varphi$
are identity mappings.

Compared to Theorems~\ref{thm:psi} and \ref{thm:psi2},
we suppose that $\Psi$ satisfies a~subadditivity type property 
\begin{equation}\label{kovaey}
\Psi(x+y)\leq \Psi(x)+\Psi(y)+O(1), 
\end{equation}
which is particularly true if $\Psi(x)=x$ or $\Psi(x)=\log^+ x$ corresponding to the usual order and the hyper order, respectively. 
In fact, if $\Psi$ is a~positive function such that $\Psi(x)/x$ is eventually
non-increasing, then $\Psi$ satisfies this subadditivity type property.
This can be proved by writing $\Psi(x)=x \cdot (\Psi(x)/x)$, where $x$
is subadditive.
The auxiliary function $\phi$ gives us freedom to apply
the definition of $(\Psi,\varphi)$-order to different growth scales.
Since $T(r,fg)\leq 2\max\{T(r,f),T(r,g)\}$ and $T(r,f+g)\leq 2\max\{T(r,f),T(r,g)\}+\log 2$ for any meromorphic $f$ and $g$,
we conclude
\begin{equation}\label{W1}
  \begin{split}
    \rho_{\Psi,\varphi}(fg) &\leq \max\{\rho_{\Psi,\varphi}(f),\rho_{\Psi,\varphi}(g)\},\\
    \rho_{\Psi,\varphi}(f+g) &\leq \max\{\rho_{\Psi,\varphi}(f),\rho_{\Psi,\varphi}(g)\}.
  \end{split}
\end{equation}

Let $s:[0,R)\to [0,R)$ be an increasing function such that $s(r)\in (r,R)$
for $0\leq r<R$. 
Using the Gol'dberg-Grinshtein estimate \cite[Corollary~3.2.3]{CY:year}, we obtain
\begin{equation}\label{deri}
  T(r,f')\lesssim 1+\log^+\frac{s(r)}{r(s(r)-r)}
  +T(s(r),f).
\end{equation}
Suppose that $\varphi$ and $s$ are chosen such that 
\begin{equation*} 
    \limsup_{r\to R^-} \, \frac{\log\varphi(s(r))}{\log\varphi(r)} = 1
\end{equation*}
and
\begin{equation} \label{eq:eeee}
    \rho_{\Psi,\varphi}\left(\log^+\frac{s(r)}{r(s(r)-r)}\right) = 0.
    \end{equation}
Then
\begin{equation}\label{W2}
  \rho_{\Psi,\varphi}(f') \leq \rho_{\Psi,\varphi}(f).
\end{equation}
The condition~\eqref{eq:eeee} is trivial for standard choices in the plane and in the disc $D(0,R)$, respectively.

The validity of the reverse inequality $\rho_{\Psi,\varphi}(f) \leq \rho_{\Psi,\varphi}(f')$ is
based on similar discussions as above and on the estimate
\begin{equation*}
  T(r,f)\lesssim \frac{s(r)}{s(r)-r}\left(\log \frac{2s(r)}{s(r)-r}\right)\big(T(s(r),f')+1\big)+\log^+ r
\end{equation*}
by Chuang \cite{Chuang}.
Regarding our applications, this reverse estimate is not needed.

Theorem~\ref{limsup-thm} below generalizes the main results in  \cite{CHR:2009} and \cite{HR:2011} to some extent.


\begin{theorem}\label{limsup-thm}
Suppose that $\Psi$, $\varphi$ and $s$ are functions as above such that \eqref{W1} and \eqref{W2} hold, but \eqref{eq:eeee} is replaced by the~stronger condition 
    \begin{equation} \label{eq:eee}
    \rho_{\Psi,\varphi}\left(\frac{s(r)}{r} \, \log\frac{e\, s(r)}{s(r)-r}\right) = 0.
    \end{equation}
In addition, we suppose $\rho_{\Psi,\varphi}(\log^+ r)=0$ and $\Psi(\log x) = o(\Psi(x))$ as $x\to\infty$.
Let $\lambda\geq 0$.
If the coefficients $A_0,\dotsc,A_{k-2}$ are analytic in $D(0,R)$, then the following conditions are equivalent:
\begin{itemize}[leftmargin=0.75cm]
\item[\rm (i)] $\rho_{\Psi,\varphi}\left(\frac{1}{r} \int_{D(0,r)}|A_j(z)|^\frac{1}{k-j}\, dm(z)\right)\leq\lambda$ for all $j=0,\ldots,k-2$;
\item[\rm (ii)] $\rho_{\Psi,\varphi}(f)\leq\lambda$ for all solutions $f$ of \eqref{ldek_gen};
\item[\rm (iii)] $\lambda_{\Psi,\varphi}(0,f)\leq\lambda$
and $\rho_{\Psi,\varphi}(f)< \infty$
for all non-trivial solutions $f$ of \eqref{ldek_gen}.
\end{itemize}
Moreover, if there exists a function for which the equality holds in any of the three
inequalities above, then there exist appropriate functions such that the equalities hold in the
remaining two inequalities.
\end{theorem}

Note the following observations regarding Theorem~\ref{limsup-thm}.

(a) Assumption~\eqref{eq:eee} restricts the possible values of $s(r)$. It requires that
$s(r)$ cannot be significantly larger than $r$, and at the same time, $s(r)-r$ cannot
be too small. For example, the choices $s(r)=c r$ and $s(r)=r(\log r)^\alpha$ are allowed
in the classical setting of the complex plane for any $c>1$ and $\alpha>0$.

(b) The assumption $\rho_{\Psi,\varphi}(\log^+ r)=0$
is trivial if $R<\infty$, while if $R=\infty$ it
is equivalent to saying that all rational functions
are of $(\Psi,\varphi)$-order zero.

(c)
By a~careful inspection of the proof of Theorem~\ref{limsup-thm}, we see
that the assumptions can be significantly relaxed if 
the quantities in (i), (ii) and (iii) are required to be simultaneously either finite or infinite. 
First, \eqref{kovaey}
can be
relaxed to $\Psi(x+y) \lesssim \Psi(x) + \Psi(y) + 1$, which is 
satisfied, for instance, by 
$\Psi(x)=x^\alpha$ for $\alpha>1$.
Then analogues of \eqref{W1} and \eqref{W2} hold, where the inequality sign $\leq$ is replaced by $\lesssim$.
Second, instead of \eqref{eq:eee} and $\rho_{\Psi,\varphi}(\log^+r)=0$, it suffices to require that the orders in question are finite. In this case the $\rho_{\Psi,\varphi}$-order can be chosen to be the logarithmic order in the finite disc and in the complex plane.


\section{Results on logarithmic derivatives} \label{sec:logdva}

Our results on differential equations are based on new estimates on logarithmic derivatives of meromorphic functions.


\begin{theorem} \label{thm:nn}
Let $0<\varrho<\infty$
and  $f\not\equiv 0$ meromorphic in
a~domain containing $\overline{D(0,\varrho)}$. Then there exists
a~positive constant $C$,
which depends only on the initial values of $f$ at the origin, such that
\begin{equation*}
  \begin{split}
    &  \int_{r'<|z|<r}  \left| \frac{f'(z)}{f(z)}  \right| \, dm(z)\\
    & \qquad \lesssim \Bigg( 4 \varrho \frac{r-r'}{\varrho-r'} \left( 2 + \log 2 + \log \frac{\varrho-r'}{r-r'} \right) 
    +  (2\pi + 2)(r-r') + 3\varrho \log\frac{\varrho-r'}{\varrho-r} \Bigg) \\
    & \qquad \qquad \times \big( 2 T(\varrho,f) + C \big),
    \quad 0\leq r' < r<\varrho.
  \end{split}
\end{equation*}
\end{theorem}

The term
\begin{equation*}
  \frac{r-r'}{\varrho-r'} \left( 2 + \log 2 + \log \frac{\varrho-r'}{r-r'} \right) 
\end{equation*}
appearing in Theorem~\ref{thm:nn}
is uniformly bounded above by $2+ \log{2}$ for all $0\leq r'<r<\varrho$, and it decays to zero as~$r'\to r$.
Therefore Theorem~\ref{thm:nn} yields
\begin{equation} \label{eq:nice}
  \begin{split}
    \int_{r'<|z|<r}  \left| \frac{f'(z)}{f(z)}  \right| \, dm(z)
    & \lesssim \varrho \log\frac{e(\varrho-r')}{\varrho-r} \, \big( T(\varrho,f) + 1 \big),
    \quad 0\leq r'<r<\varrho.
  \end{split}
\end{equation}
The following examples illustrate the sharpness of \eqref{eq:nice}.


\begin{example}
Let $f(z)=\exp(z^n)$ for $z\in\C$, and $\varrho=2r$. By a straight-forward computation, $T(r,f)=r^n/(n\pi)$ for
$0<r<\infty$. Now
\begin{equation*}
  \int_{|z|<r} \bigg| \frac{f'(z)}{f(z)} \bigg| \, dm(z) =  2\pi n  \int_0^r t^n \, dt = \frac{2\pi n}{n+1} \, r^{n+1},
  \quad 0<r<\infty,
\end{equation*}
while
\begin{equation*}
  \varrho \, \log\frac{e\varrho}{\varrho-r} \, \left(  T(\varrho,f) + 1\right)
  = 2r \big( 1+ \log 2 \big) \left( \frac{2^n r^n}{n\pi} + 1\right),
\quad 0<r<\infty.
\end{equation*}
This shows that the leading $\varrho$ in \eqref{eq:nice} cannot be removed.
\end{example}


\begin{example}
Let $f(z)=\exp(-(1+z)/(1-z))$ for $z\in\D$, and $\varrho=(1+r)/2$. By a~straight-forward computation, 
$T(r,f)=0$ for $0<r<1$. Now
\begin{equation*}
  \int_{|z|<r} \bigg| \frac{f'(z)}{f(z)} \bigg| \, dm(z) 
  = \int_{|z|<r} \frac{2}{|1-z|^2} \, dm(z) 
  = 2\pi \log\frac{1}{1-r^2}, \quad 0<r<1,
\end{equation*}
while
\begin{equation*}
  \varrho \, \log\frac{e\varrho}{\varrho-r} \, \left(  T(\varrho,f) + 1\right)
  = \frac{1+r}{2} \, \log \frac{e(1+r)}{1-r}, \quad 0<r<1.
\end{equation*}
This shows that the logarithmic term in \eqref{eq:nice} cannot be removed.
\end{example}

In the special case when $\rho/ r'$ is uniformly bounded an equivalent 
estimate (up to a constant factor) is obtained in \cite{Bar} and \cite{ChyKol}.
In fact, a~much more general class of functions is considered in \cite{ChyKol}. 
These results imply
\begin{equation*}
  \int_{r'<|z|<r} \left|\frac{f'(z)}{f(z)} \right| \, dm(z) \lesssim  \rho \, \frac{r-r'}{\rho-r} \, T(\rho, f),
  \quad 0\leq r'<r<\varrho.
\end{equation*}
On the other hand,  Gol'dberg and  Strochik \cite[Theorem 7]{GoStr85} established a general upper 
estimate for the integral of the logarithmic derivative over a region of the form 
$\{te^{i\varphi}: r'<t<r, \, \varphi\in E(t)\}$, where $E=E(t)$ is a measurable subset of 
$[0, 2\pi]$ with $m(E)\le \theta\in (0, 2\pi]$. 
This estimate  allows 
arbitrary values $r'<r<\rho$, and takes into account the measure of $E$.
Nevertheless, if  $\rho/r'$ tends to infinity, $r\asymp r'$ and $\mes E=2\pi$, then
Theorem~\ref{thm:nn} improves all known results giving
\begin{equation*}
  \int_{r'<|z|<r} \left|\frac{f'(z)}{f(z)} \right| \, dm(z) \lesssim    (r  -r') \, T(\rho, f).
\end{equation*}

We proceed to consider two consequences of Theorem~\ref{thm:nn}, the first of which
concerns generalized logarithmic derivatives.


\begin{corollary} \label{cor:nn}
Let $0<R<\infty$ and $f$
meromorphic in a~domain containing $\overline{D(0,R)}$.
Suppose that $j,k$ are integers with $k>j\geq 0$, and $f^{(j)}\not\equiv 0$.
Then
\begin{align*}
  &  \int_{r'<|z|<r}  \bigg| \frac{f^{(k)}(z)}{f^{(j)}(z)}  \bigg|^{\frac{1}{k-j}} \, dm(z) \notag\\ 
  & \qquad \lesssim R \, \log\frac{e\, (R-r')}{R-r} \, \left(  1 + \log^+ \frac{1}{R-r} + T(R,f)\right)
\end{align*}
for $0\leq r' < r<R$.
\end{corollary}

A~standard reasoning based on Borel's lemma transforms $R$ back to $r$. 
In the case of $\D$, the choice $R = r + (1-r)/T(r,f)$ implies 
\begin{equation*}
  T(R,f) \leq 2 T(r,f)\quad \text{and} \quad \log\frac{eR}{R-r} = \log \left( e + \frac{er \, T(r,f)}{1-r} \right),
\end{equation*}
the~inequality being valid outside of a~possible exceptional set $E\subset [0,1)$
such that $\int_E dr/(1-r) < \infty$. In the case of $\C$, the choice $R=r+1/(eT(r,f))$ implies
\begin{equation*}
T(R,f) \leq 2 T(r,f) \quad \text{and} \quad 
  \log\frac{eR}{R-r} 
  = \log \big( e + er \, T(r,f) \big),
\end{equation*}
the~inequality being valid outside a~possible exceptional set $E\subset [0,\infty)$ such that $\int_E dr < \infty$.

The following consequence of Theorem~\ref{thm:nn} generalizes 
\cite[Theorem~5]{CHR:2009} to an~arbitrary auxiliary function $s(r)\in(r,R)$.
A~similar result for subharmonic functions in the plane is obtained in \cite{ChyKol}; see also \cite[Lemma 5]{HayMil}.


\begin{corollary} \label{cor:psi}
Let $f$ be meromorphic in $D(0,R)$ for $R<\infty$, and
let $j,k$ be integers with $k>j\geq 0$ such that $f^{(j)}\not\equiv 0$.
Let $s:[0,R) \to [0,R)$ be an~increasing continuous function such
that $s(r) \in (r,R)$ and $s(r)-r$ is decreasing.
If $\delta\in (0,1)$, then
there exists a~measurable set $E\subset [0,R)$ with 
\begin{equation*}
  \overline{d}(E) = \limsup_{r\to R^-} \, \frac{m(E \cap [r,R))}{R-r} \leq \delta
\end{equation*}
such that
\begin{equation} \label{eq:lll}
  \int_0^{2\pi} \bigg| \frac{f^{(k)}(re^{i\theta})}{f^{(j)}(re^{i\theta})} \bigg|^{\frac{1}{k-j}} \, d\theta
  \lesssim \frac{T(s(r),f) -\log(s(r)-r)}{s(r)-r}, \quad r\in [0,R) \setminus E.
\end{equation}
Moreover, if $k=1$ and $j=0$, then the logarithmic term in \eqref{eq:lll} can be omitted.
\end{corollary}

To proof of Corollary~\ref{cor:psi} can easily be modified to obtain
the following result.

\begin{corollary} 
Let $f$ be meromorphic in $\C$, and
let $j,k$ be integers with $k>j\geq 0$ such that $f^{(j)}\not\equiv 0$.
Let $S:[0,\infty) \to [0,\infty)$ be an~increasing continuous function such
that $S(r) \in (r,\infty)$ and $S(r)-r$ is decreasing.
If $\delta\in (0,1)$, then
there exists a~measurable set $E\subset [0,\infty)$ with 
	$$
	\overline{D}(E)=\limsup_{r\to\infty}\frac{m(E \cap [0,r))}{r}\leq \delta
	$$ 
such that
\begin{equation}\label{kaniin}
  \int_0^{2\pi} \bigg| \frac{f^{(k)}(re^{i\theta})}{f^{(j)}(re^{i\theta})} \bigg|^{\frac{1}{k-j}} \, d\theta
  \lesssim \frac{T(S(r),f) +\log S(r)-\log(S(r)-r)}{S(r)-r},
\end{equation}
for $r\in [0,\infty) \setminus E$.
Moreover, if $k=1$ and $j=0$, then the logarithmic terms in \eqref{kaniin} can be omitted.
\end{corollary}


\section{Proof of Theorem~\ref{thm:nn}} \label{sec:logdva_proof}

As is the case with usual estimates for logarithmic derivatives, the proof begins with the standard differentiated
form of the Poisson-Jensen formula. Differing from the proof of \cite[Theorem~5]{CHR:2009},
where the integration is conducted in a~sequence of annuli of fixed hyperbolic width, we consider 
a~single annulus of arbitrary width in several steps. 
This is due to an~arbitrary $s(r)$, as opposed to a~specific $s(r)=1-\beta (1-r)$, $\beta\in(0,1)$, in \cite[Theorem~5]{CHR:2009}.

By the Poisson-Jensen formula,
\begin{equation*}
\begin{split}
  \log{|f(z)|} & = \frac{1}{2\pi} \int_0^{2\pi} \log |f(\varrho e^{i\varphi})| \, K(z,\varrho e^{i\varphi}) \, d\varphi\\
  & \qquad - \sum_{|a_\mu | < \varrho} \log \, \left| \frac{\varrho^2-\overline{a}_\mu z}{\varrho(z-a_\mu)} \right|
  + \sum_{|b_\nu | < \varrho} \log \, \left| \frac{\varrho^2-\overline{b}_\nu z}{\varrho(z-b_\nu)} \right|,
  \quad z\in D(0,\varrho),
\end{split}
\end{equation*}
where $\{a_\mu\}$ and $\{b_\nu\}$ are the zeros and the poles of $f$, and
\begin{equation*}
  K(z,\varrho e^{i\varphi}) 
  = \frac{\varrho^2 - |z|^2}{|\varrho e^{i\varphi} - z|^2} 
  = \Real \!\left( \frac{\varrho e^{i\varphi} +z }{\varrho e^{i\varphi} - z} \right), \quad z\in D(0,\varrho),
\end{equation*}
is the Poisson kernel. By differentiation,
\begin{equation*}
  \begin{split}
    \frac{f'(z)}{f(z)} & = \frac{1}{2\pi} \int_0^{2\pi} \log |f(\varrho e^{i\varphi})| \, 
    \frac{2\varrho e^{i\varphi}}{(\varrho e^{i\varphi}-z)^2} \, d\varphi\\
    & \qquad - \sum_{|a_\mu | < \varrho} \frac{|a_\mu|^2-\varrho^2}{(z-a_\mu)(\varrho^2-\overline{a}_\mu z)}
    + \sum_{|b_\nu | < \varrho} \frac{|b_\nu|^2-\varrho^2}{(z-b_\nu)(\varrho^2-\overline{b}_\nu z)}
  \end{split}
\end{equation*}
for all $z\in D(0,\varrho)$. Let $\{ c_m \} = \{a_\mu \} \cup \{ b_\nu\}$. We deduce
\begin{equation*}
  \left| \frac{f'(z)}{f(z)}  \right| \leq \frac{\varrho}{\pi} \int_0^{2\pi} \frac{\big| 
    \log |f(\varrho e^{i\varphi})|\big|}{|\varrho e^{i\varphi}-z|^2} \, d\varphi
  + \sum_{|c_m | < \varrho} \frac{\varrho^2 - |c_m|^2}{|z-c_m|\, |\varrho^2-\overline{c}_m z|},
  \quad z\in D(0,\varrho),
\end{equation*}
and therefore an application of Fubini's theorem yields
\begin{align}
  & \int_{r'<|z|<r}  \left| \frac{f'(z)}{f(z)}  \right| \, dm(z)\notag\\
  & \qquad \leq \frac{\varrho}{\pi} \int_0^{2\pi} \big| \log |f(\varrho e^{i\varphi})| \big|
  \left( \, \int_{r'<|z|<r} \frac{dm(z)}{|\varrho e^{i\varphi}-z|^2} \right)  d\varphi \label{eq:rhs}\\
  & \qquad \qquad + n(0) \int_{r'<|z|<r}\frac{dm(z)}{|z|}\notag\\
  &  \qquad \qquad + \sum_{0<|c_m | < \varrho} \frac{\varrho^2-|c_m|^2}{|c_m|} \,  
  \int_{r'<|z|<r} \frac{dm(z)}{|z-c_m|\, |z-\varrho^2/\overline{c}_m|}, \notag
\end{align}
where $n(r)$ is the non-integrated counting function for $c_m$-points in $|z|\leq r$, 
while~$N(r)$ is its integrated counterpart. Let $I_1$ be the 
integral in
\eqref{eq:rhs}, and let~$I_2$ be the remaining part of the upper bound.

We proceed to study $I_1=I_1(r',r,\varrho)$ and $I_2=I_2(r',r,\varrho)$ separately.
By the well-known properties of the Poisson kernel,
\begin{equation*}
  \int_{r'<|z|<r} \frac{dm(z)}{|\varrho e^{i\varphi}-z|^2}
  = 2\pi \int_{r'}^r \frac{s\, ds}{\varrho^2-s^2}
  = \pi \, \log \frac{\varrho^2-(r')^2}{\varrho^2-r^2},
\end{equation*}
and therefore
\begin{align*}
  I_1 & = \varrho \, \log \frac{\varrho^2-(r')^2}{\varrho^2-r^2}  \int_0^{2\pi} \big| \log |f(\varrho e^{i\varphi})| \big| \, d\varphi\\
  & = \varrho \, \log \frac{\varrho^2-(r')^2}{\varrho^2-r^2} \left( \,
    \int_0^{2\pi} \log^+ |f(\varrho e^{i\varphi})| \, d\varphi
    +  \int_0^{2\pi} \log ^+ \frac{1}{|f(\varrho e^{i\varphi})|} \, d\varphi \right)\\
  & \leq \varrho \, \log \frac{\varrho^2-(r')^2}{\varrho^2-r^2} \big( 2\,  T(\varrho, f) + O(1) \big).
\end{align*}
Here $O(1)$ is a~bounded term, 
which depends on the initial values of $f$ at the origin
and which arises from the application of Nevanlinna's first main theorem.

To estimate $I_2$, it suffices to find an~upper bound for
\begin{equation} \label{eq:bound}
  \int_{r'<|z|<r} \frac{dm(z)}{|z-c|\, |z-\varrho^2/c|}, \quad 0<c<\varrho.
\end{equation}
The remaining argument is divided in separate cases. Before going any further, we 
consider two auxiliary results that will be used to complete the proof of the theorem.


\begin{lemma} \label{lemma:pq}
Let $0\leq s_1\leq s_2<1$ and $0<p,q<\infty$. Then
\begin{equation*}
  J(s_1,s_2) :=\int_0^{2\pi} \frac{d\theta}{|1-s_1 e^{i\theta}|^p \, |1-s_2 e^{i\theta}|^q}
\end{equation*}
has the following asymptotic behavior:
\begin{enumerate}[leftmargin=0.75cm]
\itemsep0.5em 
\item[\rm (i)]
If $q>1$, then $J(s_1,s_2) \asymp  \displaystyle\frac{1}{(1-s_1)^p(1-s_2)^{q-1}}$;

\item[\rm (ii)]
if $q=1$, then $J(s_1,s_2) \asymp  \displaystyle\frac{1}{(1-s_1)^p} \left( \log\frac{1-s_1}{1-s_2} + 1 \right)$;

\item[\rm (iii)]
if $0<q<1$, then $J(s_1,s_2) \asymp  \displaystyle\frac{1}{(1-s_1)^{p+q-1}}$.
\end{enumerate}
\end{lemma}


\begin{proof}
Without any loss of generality, assume $1/2 \leq s_1\leq s_2<1$.
By utilizing the first three non-zero terms of cosine's Taylor series expansion, we obtain
\begin{equation*}
|1-s e^{i\theta}|^2 = 1 + s^2 - 2s\cos\theta \geq (1-s)^2+\frac{11}{12} \, s \theta^2, \quad 0<\theta<1.
\end{equation*}
The asymptotic behavior of $J(s_1,s_2)$ is comparable to that of
\begin{align*}
  & \int_0^{1} \frac{d\theta}{|1-s_1 e^{i\theta}|^p \, |1-s_2 e^{i\theta}|^q} \\
  & \qquad \leq   \Bigg[ \int_0^{1-s_2} \!+ \int_{1-s_2}^{1-s_1} \!+ \int_{1-s_1}^1 \Bigg] 
  \!\frac{d\theta}{\left(  (1-s_1)^2+\frac{11}{12}  s_1 \theta^2 \right)^{\frac{p}{2}} 
    \!\left(  (1-s_2)^2+\frac{11}{12}  s_2 \theta^2 \right)^{\frac{q}{2}}}\\
  & \qquad \lesssim \frac{1-s_2}{(1-s_1)^p(1-s_2)^q}
  + \frac{1}{(1-s_1)^p} \int_{1-s_2}^{1-s_1} \frac{d\theta}{\theta^q}
  + \int_{1-s_1}^1 \frac{d\theta}{\theta^{p+q}},
\end{align*}
which has to be estimated in the cases (i)--(iii). The details are left to the reader.
For the converse asymptotic inequality, take only the first two non-zero terms of cosine's Taylor series expansion, and
repeat the argument. 
\end{proof}


\begin{lemma} \label{lemma:intest}
Let $1\leq a < b \leq \infty$. Then
\begin{align*} 
  \int_a^b \frac{\log t}{t(t-1)}\, dt & \leq \lim_{t\to b^-} \frac{t-a}{at} \, (2+\log a),\\
  \int_a^b \frac{\log t}{t^2}\, dt & \leq \lim_{t\to b^-} \frac{t-a}{at} \, (1+\log a). 
\end{align*}
\end{lemma}


\begin{proof}
We prove the former integral estimate and leave the latter to the reader.
Let $1<b<\infty$. Then
\begin{align*}
  \int_a^b \frac{\log t}{t(t-1)}\, dt 
  & \leq \int_a^b \frac{1+\log t}{t^2} \, dt
  = \frac{\log a}{a} + \frac{2}{a} - \frac{\log b}{b} - \frac{2}{b}\\
  &  = \frac{b-a}{ab} \left( 2 + \frac{b \log a - a \log b}{b-a} \right)
  \leq \frac{b-a}{ab} \, (2+\log a).
\end{align*}
The case $b=\infty$ is an immediate modification of the above.
\end{proof}

With the help of Lemmas~\ref{lemma:pq} and \ref{lemma:intest}, we return to the proof of 
Theorem~\ref{thm:nn} and continue to estimate $I_2$.


\subsection*{Case $0\leq r' < r \leq  c < \varrho$}
Denote $x=c/\varrho$ for short.
By a change of variable, the integral in \eqref{eq:bound} can be transformed into
\begin{equation*}
  \int_{\frac{r'}{\varrho} < |w| < \frac{r}{\varrho}} \frac{dm(w)}{|w-x| \, |w-1/x|}
  = \int_{r'/\varrho}^{r/\varrho} \left( \int_0^{2\pi} \frac{d\theta}{\big|1-\frac{s}{x} e^{i\theta}\big| \, | 1-sx e^{i\theta}| }\right) s\, ds.
\end{equation*}
Let $t(s) = (1-sx)/(1-s/x)$, and note that $t$ is increasing for $s\in [0,x)$. Therefore $t(s) \geq 1$ for all $s\in [0,x)$.
By Lemma~\ref{lemma:pq}, we deduce
\begin{align*}
  \int_{r'<|z|<r} \frac{dm(z)}{|z-c|\, |z-\varrho^2/c|} 
  & \asymp \int_{r'/\varrho}^{r/\varrho}  \frac{s}{1-sx} \left( \log\frac{1-sx}{1-s/x} + 1 \right)  ds \\
  & \leq \frac{r}{\varrho} \int_{r'/\varrho}^{r/\varrho}  \frac{1}{1-sx} \left( \log\frac{1-sx}{1-s/x} + 1 \right)  ds \\
  & = \frac{rc}{\varrho^2} \int_{t(r'/\varrho)}^{t(r/\varrho)} \frac{\log t}{t(t-x^2)} \, dt 
  + \frac{r}{c} \, \log\frac{\varrho^2-c r'}{\varrho^2-c r}.
\end{align*}
An application of Lemma~\ref{lemma:intest} yields
\begin{align*}
  & \int_{r'<|z|<r} \frac{dm(z)}{|z-c|\, |z-\varrho^2/c|}\\
  & \qquad \lesssim  \frac{rc}{\varrho^2} \cdot \frac{t(r/\varrho)-t(r'/\varrho)}{t(r/\varrho) \, t(r'/\varrho)}  \left( 2+\log t(r'/\varrho) \right)
  + \frac{r}{c} \, \log\frac{\varrho^2-c r'}{\varrho^2-c r} \\
  & \qquad =  \frac{r (\varrho^2-c^2) (r-r')}{(\varrho^2-cr')(\varrho^2-cr)}  \left( 2+\log \frac{c(\varrho^2-cr')}{\varrho^2(c-r')} \right)
  + \frac{r}{c} \, \log\frac{\varrho^2-c r'}{\varrho^2-c r} \\
  & \qquad \leq \frac{c}{\varrho} \cdot \frac{r-r'}{\varrho-r'}  \left( 2+\log 2 + \log \frac{\varrho-r'}{r-r'} \right)
  + \frac{r}{c} \, \log\frac{\varrho^2-c r'}{\varrho^2-c r}.
\end{align*}


\subsection*{Case $0\leq r' \leq  c< r < \varrho$}
We write
\begin{align*}
  \int_{r'< |z| < r} \frac{dm(z)}{|z-c| \, | z - \varrho^2/c|}
  & = \int_{r'/\varrho}^x \left( \int_0^{2\pi} \frac{d\theta}{\big|1-\frac{s}{x} e^{i\theta}\big| \, |1-sxe^{i\theta}|} \right) s\, ds\\
  & \qquad + x \int_x^{r/\varrho} \left( \int_0^{2\pi} \!\frac{d\theta}{\big|1-\frac{x}{s} e^{i\theta}\big| \, |1-sxe^{i\theta}|} \right)  ds.
\end{align*}
The first integral is estimated similarly as in the case above:
\begin{align*}
  & \int_{r'/\varrho}^x \left( \int_0^{2\pi} \frac{d\theta}{\big|1-\frac{s}{x} e^{i\theta}\big| \, |1-sxe^{i\theta}|} \right) s\, ds\\
  & \qquad \lesssim \frac{c^2}{\varrho^2} \int_{t(r'/\varrho)}^{\infty} \frac{\log t}{t(t-x^2)} \, dt + \log\frac{\varrho^2-c r'}{\varrho^2-c^2}\\
  & \qquad \leq \frac{c}{\varrho} \cdot 
  \underbrace{\frac{c-r'}{\varrho-r'}  \left( 2+\log 2 + \log \frac{\varrho-r'}{c-r'} \right)}_{\text{increasing in $c\in(r',r)$}}
  + \log\frac{\varrho^2-c r'}{\varrho^2-c^2} \\
  & \qquad \leq \frac{c}{\varrho} \cdot \frac{r-r'}{\varrho-r'}  \left( 2+\log 2 + \log \frac{\varrho-r'}{r-r'} \right)
  + \log\frac{\varrho^2-c r'}{\varrho^2-c^2}.
\end{align*}
To the second integral, we apply Lemma~\ref{lemma:pq} and obtain
\begin{equation*}
  x \int_x^{r/\varrho} \left( \int_0^{2\pi} \!\frac{d\theta}{\big|1-\frac{x}{s} e^{i\theta}\big| \, |1-sxe^{i\theta}|} \right)  ds
  \asymp x \int_x^{r/\varrho} \!\!\frac{1}{1-sx} \left( \log\frac{1-sx}{1-x/s} + 1\right) ds,
\end{equation*}
which will be integrated in two parts. By Lemma~\ref{lemma:intest}, the first part gives
\begin{align*}
  x \int_x^{r/\varrho} \frac{1}{1-sx} \, \log\frac{1-sx}{1-x/s} ds
  & = x \int_x^{r/\varrho} \!\frac{\log s}{1-sx} \, ds + x \int_x^{r/\varrho} \!\!\frac{1}{1-sx} \, \log\frac{1-sx}{s-x} \, ds\\
  & \leq  x \int_x^{r/\varrho} \frac{1}{1-sx} \, \log\frac{1-sx}{s-x} \, ds\\
  & = \frac{c}{\varrho} \int_{\frac{\varrho^2-cr}{\varrho(r-c)}}^\infty \frac{\log t}{t (t+x)} \, dt
  \leq \frac{c}{\varrho} \int_{\frac{\varrho^2-cr}{\varrho(r-c)}}^\infty \frac{\log t}{t^2} \, dt\\
  & \leq \frac{c}{\varrho} \cdot \frac{\varrho(r-c)}{\varrho^2-cr} \left( 1 + \log \frac{\varrho^2-cr}{\varrho(r-c)} \right)\\
  & \leq \frac{c}{\varrho} \cdot \underbrace{\frac{r-c}{\varrho-c}  
    \left( 1+\log 2 + \log \frac{\varrho-c}{r-c} \right)}_{\text{decreasing in $c\in(r',r)$}}\\
  & \leq \frac{c}{\varrho} \cdot \frac{r-r'}{\varrho-r'}  \left( 1+\log 2 + \log \frac{\varrho-r'}{r-r'} \right),
\end{align*}
while the second part is
\begin{equation*}
  x \int_x^{r/\varrho} \frac{1}{1-sx} \,  ds = \log\frac{\varrho^2 - c^2}{\varrho^2-cr}.
\end{equation*}
In conclusion,
\begin{align*}
  \int_{r'< |z| < r} \frac{dm(z)}{|z-c| \, | z - \varrho^2/c|} 
  & \lesssim 2 \, \frac{c}{\varrho} \, \frac{r-r'}{\varrho-r'}  
  \left( 2+\log 2 + \log \frac{\varrho-r'}{r-r'} \right) + \log\frac{\varrho^2 - c r'}{\varrho^2-cr}.
\end{align*}


\subsection*{Case $0<c < r' < r < \varrho$}
As above, by Lemma~\ref{lemma:intest}, we deduce
\begin{align*}
  &  \int_{r'< |z| < r} \frac{dm(z)}{|z-c| \, | z - \varrho^2/c|}\\
  & \qquad \asymp x \int_{r'/\varrho}^{r/\varrho} \frac{1}{1-sx} \left( \log\frac{1-sx}{1-x/s} + 1\right) ds\\
  & \qquad\leq  x \int_{r'/\varrho}^{r/\varrho} \frac{1}{1-sx} \, \log\frac{1-sx}{s-x} \, ds +  x \int_{r'/\varrho}^{r/\varrho} \frac{1}{1-sx} \, ds\\
  & \qquad= \frac{c}{\varrho} \int_{\frac{\varrho^2-cr}{\varrho(r-c)}}^{\frac{\varrho^2-cr'}{\varrho(r'-c)}} \frac{\log t}{t (t+x)} \, dt + \log\frac{\varrho^2-cr'}{\varrho^2-cr}\\
  & \qquad\leq \frac{c}{\varrho}  \cdot \frac{\varrho(\varrho^2-c^2)(r-r')}{(\varrho^2-cr)(\varrho^2-cr')} \left( 1 + \log \frac{\varrho^2-cr}{\varrho(r-c)} \right)+ \log\frac{\varrho^2-cr'}{\varrho^2-cr}.\\
  & \qquad\leq 2 \cdot \frac{c}{\varrho}  \cdot \underbrace{\frac{r-r'}{\varrho-c} \left( 1 + \log 2 + \log \frac{\varrho-c}{r-r'} \right)}_{\text{increasing in $c\in(0,r')$}}+ \log\frac{\varrho^2-cr'}{\varrho^2-cr}\\
  & \qquad\leq 2 \cdot \frac{c}{\varrho}  \cdot \frac{r-r'}{\varrho-r'} \left( 1 + \log 2 + \log \frac{\varrho-r'}{r-r'} \right) + \log\frac{\varrho^2-cr'}{\varrho^2-cr}.
\end{align*}

The estimates from the three cases above can be combined into
\begin{equation*}
  \begin{split}
    & \int_{r'<|z|<r} \frac{dm(z)}{|z-c_m|\, |z-\varrho^2/\overline{c}_m|}\\
    & \qquad \lesssim \frac{2|c_m|}{\varrho} \cdot \underbrace{\frac{r-r'}{\varrho-r'} 
      \left( 2 + \log 2 + \log \frac{\varrho-r'}{r-r'} \right)}_{\text{$\leq \, 2+\log 2$, and decays to $0$ as $r' \to r$}} 
    + \log\frac{\varrho^2 - |c_m| r'}{\varrho^2-|c_m|r},
  \end{split}
\end{equation*}
for any $0<|c_m|<\varrho$.
This puts us in a~position to estimate $I_2$. We deduce
\begin{align*}
  I_2 & = 2\pi (r-r')\, n(0) + \sum_{\varepsilon<|c_m | < \varrho} \frac{\varrho^2-|c_m|^2}{|c_m|} \,  
  \int_{r'<|z|<r} \frac{dm(z)}{|z-c_m|\, |z-\varrho^2/\overline{c}_m|}\\
  & \lesssim  2\pi (r-r')\, n(0) + \frac{2}{\varrho} \, \frac{r-r'}{\varrho-r'} \left( 2 + \log 2 + 
    \log \frac{\varrho-r'}{r-r'} \right)  \sum_{\varepsilon<|c_m | < \varrho} \big( \varrho^2-|c_m|^2 \big)\\
  & \qquad +   \sum_{\varepsilon<|c_m | < \varrho}  \frac{\varrho^2-|c_m|^2}{|c_m|}\, \log\frac{\varrho^2 - |c_m| r'}{\varrho^2-|c_m|r},
\end{align*}
where $0<\varepsilon<\varrho$ is chosen such that there are no $c_m$-points in $D(0,\varepsilon)\setminus\{0\}$.
We write the sums as Riemann-Stieltjes integrals and then integrate by parts, which yields 
\begin{align*}
  \sum_{\varepsilon<|c_m | < \varrho} \big( \varrho^2-|c_m|^2 \big)
  & = \int_\varepsilon^\varrho \big( \varrho^2-t^2 \big) \, dn(t) \leq 2 \varrho^2 \int_\varepsilon^\varrho \frac{n(t)}{t} \, dt\\
  & =  2\varrho^2 \big( N(\varrho) - N(\varepsilon) \big)
  \leq 2 \varrho^2 \big( 2 T(\varrho,f) + O(1) \big).
\end{align*}
By using the estimate $\log x \leq x -1$, which holds for any positive $x$, we obtain
\begin{align*}
  & \sum_{\varepsilon<|c_m | < \varrho}  \frac{\varrho^2-|c_m|^2}{|c_m|}\, \log\frac{\varrho^2 - |c_m| r'}{\varrho^2-|c_m|r}\\
  & \qquad = \int_\varepsilon^\varrho \frac{\varrho^2-t^2}{t}\, \log\frac{\varrho^2 - tr'}{\varrho^2-tr} \, dn(t)\\
  & \qquad \leq 2 \int_\varepsilon^\varrho \log\frac{\varrho^2 - t r'}{\varrho^2- tr} \, n(t) \, dt
  + \int_\varepsilon^\varrho \frac{\varrho^2-t^2}{t}\, \log\frac{\varrho^2 - tr'}{\varrho^2-tr} \, \frac{n(t)}{t}\, dt\\
  & \qquad \leq 2 \varrho \log\frac{\varrho - r'}{\varrho- r} \int_\varepsilon^\varrho \frac{n(t)}{t} \, dt
  + \int_\varepsilon^\varrho \frac{\varrho^2-t^2}{t} \left( \frac{\varrho^2 - tr'}{\varrho^2-tr} - 1 \right) \, \frac{n(t)}{t}\, dt\\
  & \qquad \leq \left( 2  \varrho \log\frac{\varrho - r'}{\varrho- r} + 2 (r-r') \right) \times \big( 2 T(\varrho,f) + O(1) \big).
\end{align*}
Note that 
\begin{equation*}
  n(0) = n(0) \int_{\varrho/e}^\varrho \frac{dt}{t}
  \leq \int_{\varrho/e}^\varrho \frac{n(t)}{t} \, dt \leq 2 T(\varrho,f)+O(1).
\end{equation*}
Putting the obtained estimates together, we deduce
\begin{equation*}
  \begin{split}
    I_2 & \lesssim \left( 4 \varrho \frac{r-r'}{\varrho-r'} 
      \left( 2 + \log 2 + \log \frac{\varrho-r'}{r-r'} \right) +  (2\pi + 2)(r-r') + 2\varrho \log\frac{\varrho-r'}{\varrho-r} \right)\\
    & \qquad  \times \big( 2 T(\varrho,f) + O(1) \big),
    \quad 0\leq r'<r<\varrho.
  \end{split}
\end{equation*}
This completes the proof of Theorem~\ref{thm:nn}.


\section{Proofs of Corollaries~\ref{cor:nn} and \ref{cor:psi}} \label{sec:nn}

The following proof is a~straight-forward application of Theorem~\ref{thm:nn},
or more precisely, the estimate~\eqref{eq:nice}.


\begin{proof}[Corollary~\ref{cor:nn}]
Let $\varrho_0=r$ and $\varrho_{j+1} = (R+\varrho_j)/2$ for
$j=0,\dotsc,m-1$.
Using the estimate \eqref{deri} inductively,
we conclude
\begin{align*}
  T\big(\varrho_1, f^{(m)}\big)  
  & \lesssim 1 + \log^+ \frac{\varrho_2}{\varrho_1(\varrho_2-\varrho_1)} + T\big(\varrho_2,f^{(m-1)}\big)\\
  & \lesssim \dotsb \lesssim 1 + \log^+\frac{1}{R-r} + T(R,f)
\end{align*}
for any $m=j, \dotsc, k-1$. By H\"older's inequality and \eqref{eq:nice},
\begin{align*}
  \int_{r'<|z|<r}  \bigg| \frac{f^{(k)}(z)}{f^{(j)}(z)}  \bigg|^{\frac{1}{k-j}} \, dm(z)
  & = \int_{r'<|z|<r}  \prod_{m=j}^{k-1} \bigg| \frac{f^{(m+1)}(z)}{f^{(m)}(z)}  \bigg|^{\frac{1}{k-j}} \, dm(z)\\
  & \leq \prod_{m=j}^{k-1} \left( \int_{r'<|z|<r} \bigg| \frac{f^{(m+1)}(z)}{f^{(m)}(z)}  \bigg| \, dm(z)\right)^{\frac{1}{k-j}}\\
  & \lesssim \prod_{m=j}^{k-1} \left( \varrho_1 \log\frac{e(\varrho_1-r')}{\varrho_1-r} \, \big( T(\varrho_1,f^{(m)}) + 1 \big) \right)^{\frac{1}{k-j}}.
\end{align*}
The assertion follows by combining the obtained estimates.
\end{proof}


\begin{proof}[Corollary~\ref{cor:psi}]
We consider the case $k=1$ and $j=0$ only. The general case follows as in the proof of Corollary~\ref{cor:nn}.
Define the sequence $\{r_n\}_{n=0}^\infty$ such that $r_0=R/2$ and 
\begin{equation} \label{eq:ind}
  r_n = \frac{r_{n-1}+s(r_{n-1})}{2} = r_{n-1} + \frac{1}{2} \, \big( s(r_{n-1}) - r_{n-1}\big), \quad n\in\N.
\end{equation}
Since $\{r_n\}_{n=0}^\infty \subset [1/2,R)$ is increasing, there exists a limit $\lim_{n\to\infty} r_n = \alpha\leq R$.
Equation \eqref{eq:ind} implies  $2\alpha =\alpha + s(\alpha)$, which is possible only if $\alpha=R$.
We conclude $\lim_{n\to\infty} r_n=R$.

By \eqref{eq:nice}, we obtain
\begin{align*}
  \int_{r_{n-1}<|z|<r_n}  \left| \frac{f'(z)}{f(z)}  \right| \, dm(z)
  & \lesssim \log\frac{e\big(s(r_{n-1})-r_{n-1}\big)}{s(r_{n-1})-r_n} \, \Big( T\big(s(r_{n-1}),f\big) + 1 \Big)\\
  & \lesssim T\big(s(r_{n-1}),f\big) + 1, \quad n\in\N.
\end{align*}
Let
\begin{equation*}
  G_n = \left\{ r \in [r_{n-1},r_n) : \int_0^{2\pi} \left| \frac{f'(re^{i\theta})}{f(re^{i\theta})} \right|  d\theta 
    \geq K \, \frac{T\big(s(r_{n-1}),f\big) + 1}{ r_{n-1} (r_n-r_{n-1})} \right\}, \quad n\in\N,
\end{equation*}
where $K$ is a~positive constant defined later.
By the Chebyshev-Markov inequality,
\begin{align*}
  \int_{r_{n-1}<|z|<r_n}  \left| \frac{f'(z)}{f(z)}  \right| \, dm(z)
  & = \int_{r_{n-1}}^{r_n} \int_0^{2\pi}  \left| \frac{f'(re^{i\theta})}{f(re^{i\theta})}  \right|  d\theta \, r\, dr\\
  & \geq K \, \frac{T\big(s(r_{n-1}),f\big) + 1}{ r_{n-1} (r_n-r_{n-1})} \, \int_{G_n}\, r\, dr\\
  & \geq K \, \frac{T\big(s(r_{n-1}),f\big) + 1}{ (r_n-r_{n-1})} \, m(G_n), \quad n\in\N.
\end{align*}
Therefore $m(G_n) \lesssim K^{-1} (r_n-r_{n-1})$ for $n\in\N$. Define $E = [0,R/2) \cup \bigcup_{n\in\N} G_n$.

If $r\in [r_{n-1},r_n)$ for $n\in\N$, then
\begin{align*}
  \frac{m\big( E \cap [r,R)\big)}{R-r} 
  & \leq \frac{\sum_{k=n}^\infty m(G_k)}{R-r_n}
  \lesssim \frac{1}{K} \cdot \frac{\sum_{k=n}^\infty (r_k-r_{k-1})}{R-r_n}
  = \frac{1}{K} \cdot \frac{ R -r_{n-1}}{R-r_n}\\
  & \leq \frac{1}{K} \cdot \frac{ s(r_{n-1}) -r_{n-1}}{s(r_{n-1})-r_n} = \frac{2}{K}.
\end{align*}
Here we have used the property that $x \mapsto (x-r_{n-1})/(x-r_n)$ is decreasing
and positive for $x>r_n$. We deduce $\overline{d}(E) \leq \delta$, if $0<K<\infty$ is sufficiently large.
If $r\in [r_{n-1},r_n) \setminus E$ for $n\in\N$, then
\begin{align*}
  \int_0^{2\pi} \left| \frac{f'(re^{i\theta})}{f(re^{i\theta})} \right|  d\theta 
  \leq K \, \frac{T\big(s(r_{n-1}),f\big) + 1}{ r_{n-1} (r_n-r_{n-1})}
  \leq \frac{4K}{R}\,\frac{T\big(s(r_{n-1}),f\big) + 1}{s(r_{n-1})-r_{n-1}}.
\end{align*}
The assertion follows since $r\mapsto T(s(r),f)$ is increasing and $r \mapsto s(r)-r$ is decreasing.
\end{proof}


\section{Proof of Theorem~\ref{limsup-thm}} \label{sec:jh}

Before the proof of Theorem~\ref{limsup-thm}, we consider auxiliary results.


\begin{theorem}\label{Kim-thm}
{\rm \cite[Theorem~2.1]{K:1969}} Let $f_1,\ldots,f_k$ be linearly
independent solutions of \eqref{ldek_gen}, where
$A_0,\ldots,A_{k-2}$ are analytic in $D(0,R)$. Let
\begin{equation}\label{ratios-plane}
  y_1=\frac{f_1}{f_k},\ \ldots ,\ y_{k-1}=\frac{f_{k-1}}{f_k},
\end{equation}
and let $W_j$ be the determinant defined by
\begin{equation}\label{determinant-plane}
  W_j=\left|
    \begin{array}{cccc}
      y_1'\ & y_2'\ & \cdots \ & y_{k-1}'\\
      \vdots & \vdots && \vdots \\
      y_1^{(j-1)}\ & y_2^{(j-1)}\ & \cdots & y_{k-1}^{(j-1)}\\
      y_1^{(j+1)}\ & y_2^{(j+1)}\ & \cdots & y_{k-1}^{(j+1)}\\
      \vdots & \vdots& \ddots &\vdots\\
      y_1^{(k)}\ & y_2^{(k)}\ & \cdots & y_{k-1}^{(k)}
    \end{array}
  \right|,\quad j=1,\ldots, k.
\end{equation}
Then
\begin{equation}\label{coefficients-rep-plane}
  A_j=\sum_{i=0}^{k-j} (-1)^{2k-i}\delta_{ki}\left(\begin{array}{c} k-i\\ k-i-j\end{array}\right)
  \frac{W_{k-i}}{W_k} \frac{\left(\sqrt[k]{W_k}\right)^{(k-i-j)}}{\sqrt[k]{W_k}}
\end{equation}
for all $j=0,\ldots,k-2$, where $\delta_{kk}=0$ and $\delta_{ki}=1$ otherwise.
\end{theorem}

For a fixed branch of the $k$th root, there exists a constant
$C\in\C\setminus \{0\}$ such that
\begin{equation}\label{well-defined}
  \sqrt[k]{W_k}=\frac{1}{Cf_k},
\end{equation}
see \cite[Eq.~(2.6)]{K:1969}. This shows that $\sqrt[k]{W_k}$
is a well-defined meromorphic function in $D(0,R)$. For an alternative way to write
the coefficients $A_0,\ldots,A_{k-2}$ in terms of the
solutions of \eqref{ldek_gen}, see \cite[Proposition 1.4.7]{L:1993}.


\begin{lemma}\label{mero-coeffs}
Let $r<s(r)<R$, and let $g_1,\ldots, g_k$ be linearly independent
meromorphic solutions of the linear differential equation
\begin{equation}\label{lde-plane-preliminaryM}
  g^{(k)}+B_{k-1}g^{(k-1)}+\cdots +B_1g'+B_0g=0
\end{equation}
with coefficients $B_0,\ldots,B_{k-1}$ meromorphic in $D(0,R)$. Then
\begin{equation*}
   \int_{D(0,r)} |B_j(z)|^\frac{1}{k-j}\, dm(z)
  \lesssim s(r) \log \frac{e\, s(r)}{s(r)-r} \left( S(r)+\max_{1\leq l\leq k} T(s(r),g_l) \right),
\end{equation*}
for all $j=0,\ldots,k-1$. Here
\begin{equation*}
S(r)=1+\log^+\frac{1}{s(r)-r}.
\end{equation*}
\end{lemma}


\begin{proof}
We will follow the
reasoning used in proving \cite[Lemma~7.7]{L:1993}, originally developed by Frank and Hennekemper.
We proceed by induction, starting from the case $k=1$. Hence, we suppose that $B_0$ is meromorphic
in $D(0,R)$, and that $g'+B_0g=0$ has a non-trivial meromorphic solution $g_1$.
Then Corollary~\ref{cor:nn}, applied to $|B_0(z)|=|g_1'(z)/g_1(z)|$, gives us the assertion at once.
The more general case $g^{(k)}+B_0g=0$ with no
middle-term coefficients follows similarly.

Suppose next that we have proved the case $k=n\geq 1$. That is, we suppose that we have proved
the assertion for $n$ linearly independent meromorphic functions $g_1,\ldots,g_n$ solving
\begin{equation*}
  g^{(n)}+B_{n,n-1}g^{(n-1)}+\cdots+B_{n,1}g'+B_{n,0}g=0
\end{equation*}
with coefficients $B_{n,0},\ldots,B_{n,n-1}$ meromorphic in $D(0,R)$. Observe that the coefficients
$B_{n,0},\ldots,B_{n,n-1}$ are uniquely determined by
\begin{equation}\label{rep0}
  B_{n,j}=-W(g_1,\ldots,g_n)^{-1}W_j(g_1,\ldots, g_n),\quad j=0,\ldots,n-1,
\end{equation}
see \cite[Proposition~1.4.7]{L:1993}. Note that $W_j$ has a different meaning in Kim's result.

Consider $n+1$ linearly independent meromorphic functions $g_1,\ldots,g_n,g_{n+1}$.
Clearly, the Wronskian determinants $W(g_1,\ldots,g_n)$ and $W(g_1,\ldots,g_{n+1})$
do not vanish identically. Denote
\begin{equation}\label{definition-h}
  h_{n+1}=\left(\frac{d}{dz}\frac{W(g_1,\ldots,g_{n+1})}{W(g_1,\ldots,g_n)}\right)
  \bigg/\left(\frac{W(g_1,\ldots,g_{n+1})}{W(g_1,\ldots,g_n)}\right).
\end{equation}
Let $g$ be an arbitrary meromorphic function. Expanding $W(g_1,\ldots,g_{n+1},g)$ according to
the last column starting from the bottom right corner (which is associated with a positive sign
in the checkerboard pattern of signs for determinants), we get
\begin{equation}\label{W-eqn}
  \frac{W(g_1,\ldots,g_{n+1},g)}{W(g_1,\ldots,g_{n+1})}=g^{(n+1)}+\sum_{j=0}^nB_{n+1,j}g^{(j)},
\end{equation}
where
\begin{equation}\label{rep}
  B_{n+1,j}=-W(g_1,\ldots,g_{n+1})^{-1}W_j(g_1,\ldots, g_{n+1}),\quad j=0,\ldots,n.
\end{equation}
In particular, if $g\in\{g_1,\ldots,g_{n+1}\}$, then $W(g_1,\ldots,g_{n+1},g)\equiv 0$, and we
see from~\eqref{W-eqn} that the functions $g_1,\ldots,g_n,g_{n+1}$ are linearly independent
meromorphic solutions of the equation
\begin{equation*}
  g^{(n+1)}+\sum_{j=0}^nB_{n+1,j}g^{(j)}=0,
\end{equation*}
where the coefficients are given by \eqref{rep}.

Next we do some elementary computations with the Wronskian determinants appearing in the left-hand
side of \eqref{W-eqn}, see \cite[pp.~134--135]{L:1993}, and obtain the following representation
for the right-hand side of \eqref{W-eqn}:
\begin{align*}
  g^{(n+1)}+\sum_{j=0}^nB_{n+1,j}g^{(j)} & =  g^{(n+1)}+(B_{n,n-1}-h_{n+1})g^{(n)}\\
  & \qquad +\sum_{j=1}^{n-1}\left(B_{n,j}'+B_{n,j-1}-B_{n,j}h_{n+1}\right)g^{(j)}\\
  & \qquad  +\left(B_{n,0}'-B_{n,0}h_{n+1}\right)g.
\end{align*}
Comparing the corresponding coefficients, we deduce
\begin{equation}\label{coefficients}
  \left\{\begin{array}{rcl}
      B_{n+1,n} &=& B_{n,n-1}-h_{n+1},\\
      B_{n+1,j} &=& B_{n,j}'+B_{n,j-1}-B_{n,j}h_{n+1},\quad j=1,\ldots,n-1,\\
      B_{n+1,0} &=& B_{n,0}'-B_{n,0}h_{n+1}.
    \end{array}\right.
\end{equation}
H\"older's inequality yields
\begin{equation*}
  \begin{split}
    \int_{D(0,r)} &\left|B_{n+1,0}(z)\right|^\frac{1}{n+1}\, dm(z)\\
    &\leq \int_{D(0,r)}\left|B_{n,0}'(z)\right|^\frac{1}{n+1}\, dm(z)
    +\int_{D(0,r)}\left|B_{n,0}(z)h_{n+1}(z)\right|^\frac{1}{n+1}\, dm(z)\\
    &\leq \left(\int_{D(0,r)}\left|\frac{B_{n,0}'(z)}{B_{n,0}(z)}\right|\, dm(z)\right)^\frac{1}{n+1}
    \left(\int_{D(0,r)}\left|B_{n,0}(z)\right|^\frac{1}{n}\, dm(z)\right)^\frac{n}{n+1}\\
    &\qquad+\left(\int_{D(0,r)}\left|B_{n,0}(z)\right|^\frac{1}{n}\, dm(z)\right)^\frac{n}{n+1}
    \left(\int_{D(0,r)}\left|h_{n+1}(z)\right|\, dm(z)\right)^\frac{1}{n+1}.
  \end{split}
\end{equation*}
Using \eqref{rep0} and Corollary~\ref{cor:nn}, as well as \eqref{deri}, we get
\begin{equation*}
  \begin{split}
    \int_{D(0,r)} \bigg|\frac{B_{n,0}'(z)}{B_{n,0}(z)}\bigg|\,\,  dm(z)
    & \lesssim s(r) \log\frac{e\, s(r)}{s(r)-r} \left( S(r)+T\left(s(r),\frac{W_0}{W} \right) \right)\\
    & \lesssim s(r) \log\frac{e\, s(r)}{s(r)-r} \left( S(r)+ \max_{1\leq l\leq n} T(s(r),g_l) \right).
  \end{split}
\end{equation*}
Here we have also applied the proof of Corollary~\ref{cor:nn} by
introducing sufficiently many $\varrho_j$'s.
Analogously, from \eqref{definition-h} and Corollary~\ref{cor:nn} it follows that
\begin{equation*}
  \int_{D(0,r)} \left|h_{n+1}(z)\right|\, dm(z)
  \lesssim s(r) \log\frac{e\, s(r)}{s(r)-r} \left( S(r)+ \max_{1\leq l\leq n} T(s(r),g_l) \right).
\end{equation*}
The induction assumption applies for $B_{n,0}$, so that putting all estimates for $B_{n+1,0}$ together,
we deduce
the right magnitude of growth. The remaining coefficients $B_{n+1,j}$, $j=1,\ldots,n$,
in \eqref{coefficients} can be estimated similarly. This completes the proof
of the case $k=n+1$.
\end{proof}


\begin{proof}[Proof of Theorem~\ref{limsup-thm}]
Suppose that (i) holds.
By the growth estimates \cite[Corollary~5.3]{HKR:2004},
\begin{equation*} 
    m(r,f) \lesssim 
    \sum_{j=0}^{k-1} \int_0^r \int_0^{2\pi} |A_j(se^{i\theta})|^\frac{1}{k-j} \, d\theta ds + 1.
\end{equation*}
By subharmonicity,
\begin{align*}
   & r \int_0^r \int_0^{2\pi} |A_j(se^{i\theta})|^\frac{1}{k-j} \, d\theta ds \\
   & \qquad =
     r \int_0^{r/2} \int_0^{2\pi} |A_j(se^{i\theta})|^\frac{1}{k-j} \, d\theta ds + r \int_{r/2}^r \int_0^{2\pi} |A_j(se^{i\theta})|^\frac{1}{k-j} \, d\theta ds \\
     & \qquad \leq
      \frac{r^2}{2} \int_0^{2\pi} 
      \left|A_j\left(\frac{r}{2}e^{i\theta}\right)\right|^\frac{1}{k-j} d\theta  \cdot \frac{\int_{r/2}^r s\, ds}{\frac{3}{8} r^2}+ 2 \int_{r/2}^r \int_0^{2\pi} |A_j(se^{i\theta})|^\frac{1}{k-j} \, d\theta s ds,
\end{align*}
and therefore
\begin{align} \label{eq:grr}
     T(r,f)
    \lesssim 
    \frac{1}{r} \sum_{j=0}^{k-1} \int_{D(0,r)} |A_j(se^{i\theta})|^\frac{1}{k-j} \, dm(z) + 1.
\end{align}
The implication from (i) to (ii) follows from the properties of $\Psi$.
The implication from (ii) to (iii) is trivial because of
$\lambda_{\Psi,\varphi}(0,f)\leq\rho_{\Psi,\varphi}(f)$. 
It remains to prove that (iii) implies (i).

Let $f_1,\ldots,f_k$ be linearly independent solutions of
\eqref{ldek_gen}, and let $y_1,\ldots,y_{k-1}$ be defined by
\eqref{ratios-plane}. Let $j\in \{1,\ldots ,k-1\}$. We note that
the zeros and poles of $y_j=f_j/f_k$ are sequences with
$(\Psi,\varphi)$-exponent of convergence $\leq \lambda$ by the assumption~(iii). The same is true for the 1-points of $y_j$,
as they are precisely the zeros of $f_j-f_k$, which is also a
solution of \eqref{ldek_gen}. 
In other words,
\begin{equation} \label{eq:klk}
\max\Big\{ \lambda_{\Psi,\varphi}(0,y_j),\lambda_{\Psi,\varphi}(\infty,y_j),\lambda_{\Psi,\varphi}(1,y_j) \Big\} \leq \lambda.
\end{equation}
Suppose that $y_j(0)\neq 0,\infty,1$ and $y_j'(0)\neq 0$.
By the second main theorem of
Nevanlinna \cite[Theorem~1.4]{Ykirja} and the Gol'dberg-Grinshtein estimate \cite[Corollary~3.2.3]{CY:year}, we now have
\begin{equation} \label{eq:smt}
  \begin{split}
    T(r,y_j)\leq\  &N(r,y_j,0)+N(r,y_j,\infty)+N(r,y_j,1)\\
    &+O\left(1+\log^+\frac{s(r)}{r(s(r)-r)}+\log^+T(s(r),y_j)\right).
  \end{split}
\end{equation}
Since $\rho_{\Psi,\varphi}(f)<\infty$ for all solutions $f$ of \eqref{ldek_gen}, we deduce $\rho_{\Psi,\varphi}(y_j)<\infty$.
In fact, we prove
\begin{equation}\label{order-ratio-plane}
  \rho_{\Psi,\varphi}(y_j)\leq \lambda,\quad j=1,\ldots,k-1.
\end{equation}
Clearly we may suppose that $T(r,y_j)$ is an~unbounded
function of $r$. Since
\begin{equation*}
\begin{split}
    \frac{\Psi\big( \log\log \, T(s(r),y_j)\big)}{\log\varphi(r)}
    & = o \left(  \frac{\Psi\big( \log \, T(s(r),y_j)\big)}{\log\varphi(s(r))} \cdot
    \frac{\log\varphi(s(r))}{\log\varphi(r)} \right)
    =  o(1),
\end{split}
\end{equation*}
as $r\to\infty$,
the assertion \eqref{order-ratio-plane} follows
by \eqref{eq:eee} (or~\eqref{eq:eeee}), \eqref{eq:klk} and \eqref{eq:smt}. If
$y_j(0)\in\{0,\infty,1\}$ or $y_j'(0)=0$, then
the assertion \eqref{order-ratio-plane} follows
by standard arguments and the fact that rational
functions are of $(\Psi,\varphi)$-order zero by the assumption $\rho_{\Psi,\varphi}(\log^+ r)=0$.

It is claimed in \cite[p.~719]{K:1969} that the functions
$1,y_1,\ldots, y_{k-1}$ are linearly independent meromorphic
solutions of the differential equation
    \begin{equation*}
    y^{(k)}-\frac{W_{k-1}(z)}{W_k(z)}y^{(k-1)}+\cdots + (-1)^{k+1}\frac{W_1(z)}{W_k(z)}y'=0,
    \end{equation*}
where the functions $W_j$ are defined by
\eqref{determinant-plane}. This can be verified by restating
\cite[Proposition 1.4.7]{L:1993} with the aid of some basic
properties satisfied by Wronskian determinants \cite[Chapter~1.4]{L:1993}. 
From Lemma~\ref{mero-coeffs} we now conclude
\begin{equation*}
  \int_{D(0,r)}\left|\frac{W_i(z)}{W_k(z)}\right|^\frac{1}{k-i} dm(z)
   \lesssim s(r) \log\frac{e\, s(r)}{s(r)-r} \left( S(r)+ \max_{1\leq l\leq k-1} T(s(r),y_l) \right)
\end{equation*}
for all $i=1,\ldots, k-1$, or, in other words,
\begin{equation}\label{wronskian-ratios-plane}
  \int_{D(0,r)}\left|\frac{W_{k-i}(z)}{W_k(z)}\right|^\frac{1}{i} dm(z)
  \lesssim s(r) \log\frac{e\, s(r)}{s(r)-r} \left( S(r)+ \max_{1\leq l\leq k-1} T(s(r),y_l) \right)
\end{equation}
for all $i=1,\ldots, k-1$.
By \eqref{W1}, \eqref{W2}, \eqref{determinant-plane} and
\eqref{order-ratio-plane} it is clear that $\rho_{\Psi,\varphi}(W_k)\leq\lambda$.
Since $\sqrt[k]{W_k}$ is a well defined meromorphic function in $D(0,R)$ by \eqref{well-defined}, it
follows that $\rho_{\Psi,\varphi}(\sqrt[k]{W_k})\leq\lambda$. By
Corollary~\ref{cor:nn}, we have
\begin{equation}\label{Wk-logderivatives}
  \begin{split}
    & \int_{D(0,r)}\left|\frac{\left(\sqrt[k]{W_k}\right)^{(k-i-j)}(z)}{\sqrt[k]{W_k}(z)}\right|^\frac{1}{k-i-j} \, dm(z)\\
    & \qquad \lesssim s(r) \log\frac{e\, s(r)}{s(r)-r} \left( S(r)+  T(s(r),W_k) \right)\\
    & \qquad \lesssim s(r) \log\frac{e\, s(r)}{s(r)-r} \left( S(r)+ \max_{1\leq l\leq k-1} T(s(r),y_l) \right),
  \end{split}
\end{equation}
where $i$ and $j$ are as in \eqref{coefficients-rep-plane}. From \eqref{coefficients-rep-plane}, we deduce
\begin{equation*}
  |A_j|^\frac{1}{k-j}\lesssim \left|\frac{\left(\sqrt[k]{W_k}\right)^{(k-j)}}{\sqrt[k]{W_k}}\right|^\frac{1}{k-j}
  +\sum_{i=1}^{k-j}\left|\frac{W_{k-i}}{W_k}\right|^\frac{1}{k-j}
  \left|\frac{\left(\sqrt[k]{W_k}\right)^{(k-i-j)}}{\sqrt[k]{W_k}}\right|^\frac{1}{k-j}.
\end{equation*}

Finally, we make use of H\"older's inequality with conjugate indices $p=(k-j)/i$ and $q=(k-j)/(k-i-j)$,
$1\leq i< k-j$, ($i=k-j$ is a removable triviality) together with
\eqref{wronskian-ratios-plane} and \eqref{Wk-logderivatives}, and conclude
\begin{equation*}
    \frac{1}{r} \int_{D(0,r)} |A_j(z)|^\frac{1}{k-j} \, dm(z)
    \lesssim \frac{s(r)}{r} \log\frac{e\, s(r)}{s(r)-r} \left( S(r)+ \max_{1\leq l\leq k-1} T(s(r),y_l) \right)
\end{equation*}
for $j=0,\ldots,k-1$.
By \eqref{eq:eee}, \eqref{order-ratio-plane} and
the properties of $\Psi$ and $\varphi$, we deduce 
$$\rho_{\Psi,\varphi}\left(\frac{1}{r} \int_{D(0,r)}|A_j(z)|^\frac{1}{k-j}\, dm(z)\right)\leq\lambda, \quad j=0,\ldots,k-1.$$

We have proved that (i), (ii), (iii) are equivalent. 
Suppose that there exists an appropriate function for which the equality holds in one of these three inequalities. 
If a strict inequality holds in either of the remaining two inequalities, then a strict inequality should hold in all three, 
which is a contradiction.
\end{proof}


\section{Proof of Theorem~\ref{thm:psi}}

Note that the assumption $\Psi(x^2)\lesssim \Psi(x)$, $0\leq x < \infty$, implies
\begin{equation*}
  \Psi(xy) \lesssim \Psi(x) + \Psi(y), \quad \Psi(x + y) \lesssim \Psi(x) + \Psi(y)  + 1
\end{equation*}
for all $0\leq x,y <\infty$. The following result is a~counterpart of Lemma~\ref{mero-coeffs}.


\begin{lemma}\label{mero-coeffs2}
Suppose that $\Psi, s,\omega$ are functions as in Theorem~\ref{thm:psi}.
Let $g_1,\ldots, g_k$ be linearly independent
meromorphic solutions of a linear differential equation \eqref{lde-plane-preliminaryM}
with coefficients $B_0,\ldots,B_{k-1}$ meromorphic in $D(0,R)$. If
\begin{equation*}
  \int_0^R \Psi\big( T(r,g_j) \big) \, \omega(r) \, dr < \infty, \quad j=0,\dotsc,k,
\end{equation*}
then
\begin{equation*}
  \int_0^R \Psi\bigg( \int_{D(0,r)} |B_j(z)|^{\frac{1}{k-j}} \, dm(z) \bigg) \, \omega(r)\, dr < \infty,
  \quad j=0,\dotsc,k-1.
\end{equation*}
\end{lemma}


\begin{proof}
We only consider a~special case of \eqref{lde-plane-preliminaryM},
where all intermediate coefficients are identically zero, i.e., 
\begin{equation} \label{ldek}
  g^{(k)} + B_0 g = 0.
\end{equation}
The general case can be obtained by using the Frank-Hennekemper approach as in 
the proof of Lemma~\ref{mero-coeffs}, or by applying the standard order reduction procedure 
\cite[pp.~106--107]{P:1978}.

Let $g$ be any non-trivial meromorphic solution of \eqref{ldek}. Now
\begin{equation} \label{eq:grgr}
  \int_r^R \Psi\big( T(t,g) \big) \, \omega(t) \, dt 
  \geq \Psi\big( T(r,g) \big) \, \widehat{\omega}(r),
  \quad 0\leq r < R.
\end{equation}
Note that the left-hand side of \eqref{eq:grgr} decays to zero as $r\to R^-$.
Corollary~\ref{cor:nn} implies
\begin{equation*}
  \begin{split}
    \int_{D(0,r)} \bigg| \frac{g^{(k)}(z)}{g(z)} \bigg|^{\frac{1}{k}} \, dm(z)
    & \lesssim s(r) \log\frac{e \, s(r)}{s(r)-r} \\
    & \qquad \times \left( 1  + \log^+\!\frac{s(r)}{r(s(r)-r)} +  T\big(s(r),g \big) \right)
  \end{split}
\end{equation*}
for all $0<r<R$. Therefore, by the properties of $\Psi$, we obtain
\begin{align}
  & \int_0^R \Psi\bigg(  \int_{D(0,r)} |B_0(z)|^{\frac{1}{k}} \, dm(z)  \bigg) \, \omega(r) \, dr\notag\\
  & \qquad = \int_0^R \Psi\bigg(  
  \int_{D(0,r)} \bigg| \frac{g^{(k)}(z)}{g(z)} \bigg|^{\frac{1}{k}} \, dm(z)  \bigg) \, \omega(r) \, dr\notag\\
  & \qquad \lesssim \int_0^R \Psi\big( T(s(r),g)  \big) \, \omega(r) \, dr
  + \int_0^R \Psi\bigg( 
  s(r) \log\frac{e \, s(r)}{s(r)-r} 
  \bigg) \, \omega(r) \, dr \label{eq:tag1} +1.
\end{align}
The latter integral in \eqref{eq:tag1} is finite by 
\eqref{eq:intass}, while
the former integral is integrated by parts as follows:
\begin{align*}
  & \int_0^R \Psi\big( T(s(r),g)  \big) \, \omega(r) \, dr
= \int_{s(0)}^R \Psi\big( T(t,g)  \big) \, \omega\big(s^{-1}(t) \big) \, \big(s^{-1}\big)'(t) \, dt \\
  & \qquad = -\Psi\big( T(s(0),g)  \big)  \left( - \int_{s(0)}^R  \omega\big(s^{-1}(x) \big) \, \big(s^{-1}\big)'(x) \, dx \right) \\
  &\qquad \qquad  - \int_{s(0)}^R \left( \frac{\partial}{\partial t} \Psi\big( T(t,g)  \big) \right) 
  \left( - \int_t^R  \omega\big(s^{-1}(x) \big) \, \big(s^{-1}\big)'(x) \, dx \right)  dt \\
  & \qquad = \Psi\big( T(s(0),g)  \big) \, \widehat{\omega}(0) 
  + \int_{s(0)}^R \left( \frac{\partial}{\partial t} \Psi\big( T(t,g)  \big) \right) \widehat{\omega}\big(s^{-1}(t)\big)   \,  dt.
\end{align*}
By using the assumption on $\widehat{\omega}$ and integrating by parts again, we deduce
\begin{align*}
  & \int_0^R \Psi\big( T(s(r),g)  \big) \, \omega(r) \, dr\\
  & \qquad \lesssim \Psi\big( T(s(0),g)  \big) \, \widehat{\omega}(0) 
  + \int_{s(0)}^R \left( \frac{\partial}{\partial t} \Psi\big( T(t,g)  \big) \right) \widehat{\omega}(t)   \,  dt \\
  & \qquad \lesssim \Psi\big( T(s(0),g)  \big) \, \widehat{\omega}(0) 
  + \lim_{t\to R^-} \!\!\bigg( \!\Psi\big( T(t,g) \big) \, \widehat{\omega}(t) \!\bigg)
  + \int_{s(0)}^R \!\Psi\big( T(t,g)  \big) \, \omega(t)   \,  dt \\
  & \qquad \lesssim \Psi\big( T(s(0),g)  \big) \, \frac{\widehat{\omega}(0)}{\widehat{\omega}(s(0))} \int_{s(0)}^R \omega(t)\, dt
  + \int_{s(0)}^R \Psi\big( T(t,g)  \big) \, \omega(t)   \,  dt \\
  & \qquad \lesssim  \left( \frac{\widehat{\omega}(0)}{\widehat{\omega}(s(0))} + 1\right) 
  \int_{s(0)}^R \Psi\big( T(t,g)  \big) \, \omega(t)   \,  dt < \infty.
\end{align*}
The assertion follows.
\end{proof}


\begin{proof}[Theorem~\ref{thm:psi}]
Assume that (i) holds, and let $f$ be any solution of \eqref{ldek_gen}.
By~\eqref{eq:grr},
there exists a~constant $C=C(f)>0$ 
such that
\begin{equation*}
  \int_0^R \Psi\big( T(r,f) \big) \, \omega(r) \, dr 
  \leq \int_0^R \Psi\Bigg(  \frac{C}{r} \sum_{j=0}^{k-2} \int_{D(0,r)} |A_j(z)|^{\frac{1}{k-j}} \, dm(z) + C \Bigg) \, \omega(r) \, dr.
\end{equation*}
We deduce (ii) by the properties of $\Psi$.

Since (ii) implies (iii) trivially, we only need to prove that (iii) implies
(i). A~similar argument appears in the proof of Theorem~\ref{limsup-thm}, and therefore
we will only sketch the proof. Let $f_1,\dotsc,f_k$ be linearly independent solutions of \eqref{ldek_gen},
and define $y_j = f_j/f_k$ for $j=1,\dotsc,k$. 

Integrating by parts as in the proof of 
Lemma~\ref{mero-coeffs2}, and using 
$\Psi(\log x) = o(\Psi(x))$, 
we deduce for each $\varepsilon>0$ the existence of $r_0\in (0,R)$
such that
\begin{equation*}
    \int_{r_0}^{R_0} \Psi\big( \log T(s(r),y_j) \big)\, \omega(r)\, dr
    \leq \varepsilon
    \int_{r_0}^{R_0} \Psi\big(T(r,y_j) \big)\, \omega(r)\, dr
\end{equation*}
for all $R_0\in (r_0,R)$.
By applying the second main theorem of Nevanlinna~\eqref{eq:smt}, 
choosing an~appropriate $\varepsilon>0$ and re-organizing terms, we obtain
\begin{equation*}
  \begin{split}
    \int_{r_0}^{R_0} \Psi\big(T(r,y_j)\big) \, \omega(r)\, dr\lesssim\  & 
    \max_{\zeta\in\{0,\infty,1\}}\int_{r_0}^{R_0} \Psi\big(N(r,y_j,\zeta)\big)\, \omega(r)\, dr +1.
  \end{split}
\end{equation*}
By letting $R_0\to R$,
and applying~(iii), we deduce
\begin{equation*}
  \int_0^R \Psi\big( T(r,y_j) \big) \, \omega(r) \, dr < \infty, \quad j=1, \dotsc,k.
\end{equation*}
The condition~(i) can be deduced from Lemma~\ref{mero-coeffs2} by
an~argument similar to that in the proof of Theorem~\ref{limsup-thm}.
With this guidance, we consider  Theorem~\ref{thm:psi} proved.
\end{proof}


\section{Proof of Theorem~\ref{thm:ident}} \label{sec:mem}

The proof is similar to that of \cite[Theorem~7.9]{PR:2014}.
We content ourselves by proving the following result, which plays a~crucial role in the reasoning yielding
Theorem~\ref{thm:ident}. More precisely, it is a~counterpart of \cite[Lemma~7.7]{PR:2014}.


\begin{lemma} \label{lemma:BN}
Let $\omega\in\mathcal{D}$, and let $k>j\geq 0$ be integers. 
If $f$ is a~meromorphic function in~$\D$ such that
  $\int_0^1 T(r,f) \, \omega(r)\,  dr<\infty$,
then
\begin{equation*}
  \int_{\D} \, \bigg| \frac{f^{(k)}(z)}{f^{(j)}(z)} \bigg|^{\frac{1}{k-j}}\, \widehat{\omega}(z)\, dm(z) < \infty.
\end{equation*}
\end{lemma}


\begin{proof}
Let $\{\varrho_n\}$ be a sequence of points in $(0,1)$ such that $\varrho_0=0$ and
$\widehat{\omega}(\varrho_n)=\widehat{\omega}(0)/K^n$
for $n\in\N$. By \cite[Lemma~2.1]{P:2016},
the assumption $\omega\in\widehat{\mathcal{D}}$ is equivalent to the fact that
there exist constants $K=K(\omega)>1$ and $C=C(\omega,K)>1$ such that $1-\varrho_n \geq C (1-\varrho_{n+1})$
for all $n\in\N$. Let $K$ be fixed in such a way.
The assumption $\omega\in\widecheck{\mathcal{D}}$
is equivalent to the fact that there exists a~constant $\mu=\mu(\omega,K)>1$ such that
$1-\varrho_n \leq \mu (1-\varrho_{n+1})$ for all $n\in\N$; see, for example, 
the beginning of the proof of \cite[Theorem~7]{PR:preprint}. These properties give
\begin{align*}
  \frac{\varrho_{n+2}-\varrho_n}{\varrho_{n+2}-\varrho_{n+1}}
  = \frac{(1-\varrho_n) - (1-\varrho_{n+2})}{(1-\varrho_{n+1})-(1-\varrho_{n+2})} 
  \leq  \frac{\mu -1/\mu}{1-1/C}, \quad n\in\N.
\end{align*}
Then, by Corollary~\ref{cor:nn}, we obtain
\begin{align*}
  &  \int_{\D} \, \bigg| \frac{f^{(k)}(z)}{f^{(j)}(z)} \bigg|^{\frac{1}{k-j}} \, \widehat{\omega}(z)\, dm(z)\\
  & \qquad \leq \sum_{n=0}^\infty \, \widehat{\omega}(\varrho_n) 
  \int_{\varrho_{n}\leq |z| < \varrho_{n+1}} \bigg| \frac{f^{(k)}(z)}{f^{(j)}(z)} \bigg|^{\frac{1}{k-j}} \, dm(z)\\
  & \qquad \lesssim \sum_{n=0}^\infty \, \widehat{\omega}(\varrho_n) 
  \, \log\frac{2e(\varrho_{n+2}-\varrho_n)}{\varrho_{n+2}-\varrho_{n+1}} 
  \left( 1 + \log\frac{1}{\varrho_{n+2}-\varrho_{n+1}}  + T(\varrho_{n+2}, f) \right)\\
  & \qquad \lesssim \sum_{n=0}^\infty \, \widehat{\omega}(\varrho_n) 
  \, 
  \left( 1 + \log\frac{1}{\varrho_{n+2}-\varrho_{n+1}}  + T(\varrho_{n+2}, f) \right)
  =: S_1 + S_2 + S_3.
\end{align*}
We consider these sums separately. Now
$S_1= \widehat{\omega}(0) \, \sum_{n=0}^\infty K^{-n}<\infty$, while
\begin{align*}
  S_2 & \leq \left( \log \frac{C}{C-1} \right) \, \sum_{n=0}^\infty \widehat{\omega}(\varrho_n) 
  + \frac{K^2}{K-1} \sum_{n=0}^\infty 
  \log\frac{1}{1-\varrho_{n+1}} \int_{\varrho_{n+1}}^{\varrho_{n+2}} \omega(s) \, ds\\
  & \lesssim 1 + \int_{\varrho_1}^1 \log\frac{1}{1-s} \, \omega(s) \, ds < \infty.
\end{align*}
To see that this last integral is finite, let $r_n=1-2^{-n}$ for $n\in\N\cup\{0\}$, and compute
\begin{align*}
  \int_0^1 \log\frac{1}{1-s} \, \omega(s) \, ds
  & = \sum_{n=0}^\infty \int_{r_n}^{r_{n+1}} \log\frac{1}{1-s} \, \omega(s) \, ds \\ 
  & \leq \sum_{n=0}^\infty \log\frac{1}{1-r_{n+1}} \, \Big( \widehat{\omega}(r_n) - \widehat{\omega}(r_{n+1})\Big)    \\
  & \lesssim  \sum_{n=0}^\infty \, \log\frac{1}{1-r_{n+1} }\,  \widehat{\omega}(r_{n+1})  \\
  & \lesssim \sum_{n=1}^\infty n \, \widehat{\omega}(r_n) \lesssim \sum_{n=1}^\infty \frac{n}{K^n} < \infty.
\end{align*}

To estimate the last sum, we write
\begin{align*}
  S_3  
  & =\sum_{n=0}^\infty \, \frac{\widehat{\omega}(\varrho_n)}
  {\widehat{\omega}(\varrho_{n+2}) - \widehat{\omega}(\varrho_{n+3})} \, T(\varrho_{n+2}, f) 
  \int_{\varrho_{n+2}}^{\varrho_{n+3}} \omega(r)\, dr \\
  & \leq \frac{K^3}{K-1}  \, \sum_{n=0}^\infty \int_{\varrho_{n+2}}^{\varrho_{n+3}} T(r, f) \, \omega(r)\, dr
  \leq \frac{K^3}{K-1}  \, \int_{\varrho_{2}}^{1} T(r, f) \, \omega(r)\, dr < \infty.
\end{align*}
This completes the proof of Lemma~\ref{lemma:BN}.
\end{proof}


\end{document}